\documentclass[reqno]{amsart}

\usepackage{amsmath}

\usepackage{hyperref}
\usepackage{amssymb}
\usepackage{latexsym}

\usepackage{comment}

\usepackage{graphicx}

\usepackage{tikz}
\usetikzlibrary{matrix,arrows}

\usepackage[all]{xy}
\usepackage{setspace}

\usepackage[a4paper, hmargin=3.0cm, tmargin=3.2cm, bmargin=3.2cm]{geometry}
\numberwithin{equation}{section}

\theoremstyle{definition}
\newtheorem{defn}[equation]{Definition}
\newtheorem{rem}[equation]{Remark}

\theoremstyle{plain}
\newtheorem{lem}[equation]{Lemma}
\newtheorem{thm}[equation]{Theorem}
\newtheorem{prop}[equation]{Proposition}
\newtheorem{cor}[equation]{Corollary}
\newtheorem{claim}[equation]{Claim}
\newtheorem{observation}[equation]{Observation}
\newcommand{\bR}{\mathbb{R}}

\newcommand{\bZ}{\mathbb{Z}}

\newcommand{\cU}{\mathcal{U}}

\newcommand{\colim}{\operatorname{colim}}

\newcommand{\Diff}{\mathrm{Diff}}

\newcommand{\eps}{\epsilon}
\newcommand{\res}{\mathrm{res}}

\newcommand{\im}{\operatorname{Im}}

\newcommand{\id}{\operatorname{id}}

\newcommand{\sign}{\operatorname{sign}}

\newcommand{\Spin}{\mathrm{Spin}}

\newcommand{\twomatrix}[4]{\begin{pmatrix} #1 & #2 \\ #3 & #4  \end{pmatrix}}

\newcommand{\Riem}{{\mathcal R}}

\newcommand{\scal}{\mathrm{scal}}

\newcommand{\norm}[1]{\| #1 \|}

\newcommand{\torp}{\mathrm{tor}}

\title[The Gromov-Lawson-Chernysh surgery theorem]{The Gromov-Lawson-Chernysh surgery theorem}

\author{Johannes Ebert}
\thanks{J.E. was partially supported by the SFB 878 ``Groups, Geometry and Actions''}
\address{Mathematisches Institut, Universit\"at M\"unster\\
Einsteinstra{\ss}e 62\\
48149 M\"unster\\
Bundesrepublik Deutschland}
\email{johannes.ebert@uni-muenster.de}

\author{Georg Frenck}
\thanks{G.F. was supported by the SFB 878 ``Groups, Geometry and Actions''}
\address{Mathematisches Institut, Universit\"at M\"unster\\
Einsteinstra{\ss}e 62\\
48149 M\"unster\\
Bundesrepublik Deutschland}
\email{g\_fren01@uni-muenster.de}

\date{\today}


\begin{document}

\begin{abstract}
In this article, we give a complete and self--contained account of Chernysh's strengthening \cite{Chernysh} of the Gromov--Lawson surgery theorem \cite{GroLaw} for metrics of positive scalar curvature. No claim of originality is made.
\end{abstract}

\maketitle

\tableofcontents

\section{Introduction}

A famous result by Gromov--Lawson \cite{GroLaw} and Schoen--Yau \cite{SchYau} states that if $M^d$ is a closed manifold with a metric of positive scalar curvature and $\varphi: S^{d-k} \times \bR^{k} \to M$ a surgery datum of codimension $k \geq 3$, then the surgered manifold $M_\varphi: =M \setminus (S^{d-k} \times D^k) \cup_{S^{d-k}\times S^{k-1}} D^{d-k+1} \times S^{k-1}$ does have a metric of positive scalar curvature as well. This has been the basis for virtually all existence results for psc metrics on high-dimensional manifolds, the most prominent of which is \cite{Stolz}. 

A strengthening of the surgery theorem has been proven by Chernysh \cite{Chernysh}, based on Gromov--Lawson's proof. His result implies that the two spaces $\Riem^+ (M)$ and $\Riem^+ (M_\varphi)$ of psc metrics have the same homotopy type if in addition to $k \geq 3$ the condition $d-k+1\geq 3$ is also satisfied. %Chernysh's theorem implies that the homotopy type of $\Riem^+(M)$ is a cobordism invariant (see below for a more precise statement). 
%When \cite{Chernysh} appeared, the result was apparently perceived as a curiosity and drew little attention. This has changed in recent years: Theorem \ref{thm:main} is an irreplacable ingredient in the papers \cite{BERW} and \cite{ERW17}. 

To state Chernysh's theorems in full generality, some preliminaries are needed. In order to keep the length of this introduction at bay, we state the results somewhat informally and refer to the main body of the paper for precise definitions. 
We consider Riemannian metrics on compact manifolds with boundary (it is always assumed that the boundaries are equipped with collars). Let $\Riem(M)$ be the space of all Riemannian metrics $h$ on $M$ such that $h = g+dt^2$ near $\partial M$, for some metric $g$ on $\partial M$, and with respect to the given collar. Let $\Riem^+(M) \subset \Riem(M)$ be the subspace of metrics of positive scalar curvature. If $h\in \Riem^+ (M)$ is of the form $g+dt^2$ near $\partial M$, then $g$ has positive scalar curvature as well, and hence mapping $h$ to $g$ defines a continuous restriction map
\[
 \res: \Riem^+ (M) \to \Riem^+ (\partial M).
\]
We define 
\[
 \Riem^+ (M)_g := \res^{-1} (g),
\]
the space of all Riemannian metrics of positive scalar curvature on $M$ which near $\partial M $ are equal to $g+dt^2$.
\begin{thm}[Chernysh \cite{Chernysh2}]\label{thm:improved-chernysh-theorem}
The restriction map $\res:\Riem^+ (M) \to \Riem^+ (\partial M)$ is a Serre fibration.
\end{thm}
In fact, this is a slight improvement of the main result of \cite{Chernysh2}, where it is only shown that $\res$ is a quasifibration. The proof of Theorem \ref{thm:improved-chernysh-theorem} is given in \S \ref{subsec:fibrationtheorem} and follows largely the idea of \cite{Chernysh2}.

Now let $N$ be a compact manifold with collared boundary and let $\varphi:N \times \bR^k \to M$ be an open embedding such that $\varphi^{-1}(\partial M)=\partial N \times \bR^k$ and such that $\varphi$ is compatible with the chosen collars of $M$ and $N$. Let $g_N \in \Riem (N)$ be a Riemannian metric on $N$, not necessarily of positive scalar curvature.

Let $g_\torp^k$ be a \emph{torpedo metric} on $\bR^k$ such that $\scal (g_\torp^k) + \scal(g_N)= \scal (g_N + g_\torp^k) >0$. The precise definition of a torpedo metric will be given in \eqref{defn:torpedometrix} below, and for the time being, let us only list the most important features. Firstly, $g_\torp^k$ is an $O(k)$-invariant metric on $\bR^k$. Secondly, let $\psi: (0,\infty) \times S^{k-1} \to \bR^k$ be the polar coordinate map and $d\xi^2$ be the round metric on $S^{k-1}$. We require that $\psi^* g_\torp^k = dt^2 + \delta d\xi^2$ on $[R,\infty) \times S^{k-1}$ for some $R>0$ and $\delta>0$. Thirdly, $\scal(g_\torp^k) \geq \frac{1}{\delta^2}(k-1)(k-2)$. We define the subspace 
\[
 \Riem^+ (M,\varphi) := \{ h \in \Riem^+ (M) \vert \varphi^* h |_{N \times B_R^k}= (g_N + g_\torp^k)|_{N \times B_R^k} \} \subset \Riem^+ (M) .
\]
\begin{thm}[Chernysh \cite{Chernysh}, see also Walsh \cite{WalshC}]\label{thm:main}%[\cite{Chernysh}, \cite{Walsh13}] 
Let $\varphi : N \times \bR^k \to  M$ be an open embedding as before with $k \geq 3$. Let $g_N \in \Riem (M)$ be a Riemannian metric on $N$, not necessarily of positive scalar curvature. Let $g_\torp^k$ be a torpedo metric on $\bR^k$ so that the product metric $g_N + g_\torp^k$ on $N \times \bR^k$ has positive scalar curvature. Then the inclusion map
\[
\Riem^+ (M,\varphi) \to \Riem^+ (M)%  \text{ and } \Riem^+ (M,\varphi) \to \Riem^+ (M) % := \{ g \in \Riem^+ (M) \vert \varphi^* g |_{N \times B_R^k}= (g_N + g_\torp^k)|_{N \times B_R^k} \} 
\]
is a weak homotopy equivalence.
\end{thm}
The main bulk of this paper is devoted to a detailed discussion of the proof of Theorem \ref{thm:main}. 
\begin{rem}
What Gromov and Lawson proved is that under the hypotheses of Theorem \ref{thm:main} and for closed $N$, $\Riem^+ (M,\varphi) \neq \emptyset$, provided that $\Riem^+ (M) \neq \emptyset$. Later, Gajer \cite{Gajer} improved their result and proved that the inclusion map $\Riem^+ (M,\varphi) \to \Riem^+ (M)$ is $0$-connected. 
\end{rem}
\begin{rem}\label{rem:surgery.invariance}
If $N = S^{d-k}$ and $g_N$ is the round metric, one obtains a zig-zag
\[
\Riem^+ (M) \stackrel{\simeq}{\leftarrow} \Riem^+ (M, \varphi) \cong \Riem^+ (M_\varphi,\varphi') \to \Riem^+ (M_\varphi),
\]
where $\varphi': S^{k-1} \times \bR^{d-k+1} \to M_\varphi$ is the opposite surgery datum. It follows that $\Riem^+ (M_\varphi)\neq \emptyset$ and $\Riem^+ (M) \simeq \Riem^+ (M_\varphi)$ if $3 \leq k \leq n-2$.
\end{rem}
More generally, Theorem \ref{thm:main} implies the following cobordism invariance result.
\begin{thm}\label{thm:bordismapplication}
Let $\theta:B \to BO(d)$ be a fibration, $d \geq 6$. Assume that $M_i$, $i=0,1$, are two closed $(d-1)$-dimensional $\theta$-manifolds which are $\theta$-cobordant. Then
\begin{enumerate}
\item if the structure map $M_1 \to B$ is $2$-connected, then there is a map $\Riem^+ (M_0)\to \Riem^+ (M_1)$ (in particular, if $\Riem^+ (M_0)\neq \emptyset$, then $\Riem^+ (M_1) \neq \emptyset$).
\item If in addition the structure map $M_0 \to B$ is $2$-connected as well, then $\Riem^+ (M_0)\simeq \Riem^+ (M_1)$.
\end{enumerate}
\end{thm}

The best-known special case is $B = B \Spin (d)$. In that case, the hypothesis that $M_i\to B\Spin (d)$ is $2$-connected just means that $M_i$ is simply connected. For such manifolds, Theorem \ref{thm:bordismapplication} follows in a straightforward manner from Theorem \ref{thm:main} and the proof of the h-cobordism theorem (see e.g. \cite[Theorem VIII.4.1]{Kos}), as explained in \cite[\S 4]{WalshC}. The general case requires techniques from surgery and handlebody theory which are not so well--known, which is why we include the proof in \S \ref{sec:coboridmsection}. 

Chernysh also proved a version of Theorem \ref{thm:main} for a fixed boundary condition, which is used in an essential way in \cite{ERW17}. To state it, let $\partial \varphi: \partial N\times\bR^k \to \partial M$ be the induced embedding and let $g \in \Riem^+ (\partial M, \partial \varphi)$ be a fixed boundary condition. We let 
\[
\Riem^+ (M,\varphi)_g := \Riem^+ (M,\varphi) \cap \Riem^+ (M)_g.
\]
\begin{thm}[Chernysh \cite{Chernysh2}]\label{thm:main-boundary}%[\cite{Chernysh}, \cite{Walsh13}] 
Under the hypotheses of Theorem \ref{thm:main}, the inclusion map
\[
\Riem^+ (M,\varphi)_g \to \Riem^+ (M)_g
\]
is a weak homotopy equivalence.
\end{thm}
The proof of Theorem \ref{thm:main-boundary} is only sketched in \cite{Chernysh2}. We give a detailed proof, somewhat different from the proof envisioned in \cite{Chernysh2}, in \S \ref{sec:maintheoremboundary}. Besides Theorems \ref{thm:main} and \ref{thm:improved-chernysh-theorem}, the proof uses the (elementary) corner smoothing technique which was developped in \cite[\S 2]{ERW17}.

When \cite{Chernysh} appeared, his result was apparently perceived as a curiosity and drew little attention. This has changed in recent years: Theorem \ref{thm:main} is an irreplacable ingredient in the papers \cite{BERW} and \cite{ERW17}. 
Important parts of \cite{Chernysh} are written in a fairly obscure way, and the paper has never been published. Later Walsh published a paper \cite{WalshC} containing a proof of Theorem \ref{thm:main}, but many relevant details are not addressed in \cite{WalshC}. Because of the importance of the result for \cite{BERW} and \cite{ERW17}, the first named autor wanted to make sure that the result is correct and that he understands the proof properly. He suggested checking \cite{Chernysh} and \cite{WalshC} as a project for the second author's Master's thesis. The present paper is the result of this checking process. Let us summarize our findings. 

\begin{enumerate}
\item One half of the proof of Theorem \ref{thm:main} is virtually identical to the proof of the original Gromov--Lawson result. We found one small computational error, which is reproduced in various expositions of the result (\cite{RosSto}, \cite{Walsh11}). This error looks harmless at first sight, but enforces an alternative argument at one key juncture of the proof. %Therefore, we decided to give a self-contained treatment of the whole proof. 
\item All other arguments in Chernysh's paper are essentially correct and complete, albeit some parts of his paper are very intransparent and hard to decipher.
%\item Chernysh proves that the inclusion map is an actual homotopy equivalence instead of a weak one. This is probably the reason for the obscure style in parts of \cite{Chernysh}. The proof that the inclusion is a weak equivalence is much easier to write down and to follow. For this reason (and because there seems to be little gain by knowing that the map is an actual homotopy equivalence), we wrote down the argument for weak equivalence only. 
\item \cite{WalshC} leaves many questions open. In particular, it remains unclear to us how to fill in the details of the proof of Lemma 3.3 loc.cit., without using the quite technical computations of \cite[\S 3]{Chernysh} or computations of a similar delicacy.
\end{enumerate}

\section{Preliminary material}

\subsection{Spaces of psc metrics on manifolds with boundary}

Let $M$ be a compact manifold with boundary $\partial M$. We assume that the boundary of $M$ comes equipped with a collar $\partial M \times [0,1) \to M$. The collar identifies $\partial M \times [0,1) $ with an open subset of $M$ and we usually use this identification without further mentioning.

We only consider Riemannian metrics on $M$ which have a simple structure near $\partial M$. More precisely, for $c\in (0,1)$, we denote by $\Riem (M)^c$ the space of all Riemannian metrics $h$ on $M$ such that $h = g+dt^2$ on $\partial M \times [0,c]$ for some metric $g$ on $\partial M$. 

We topologize $\Riem(M)^c$ as a subspace of the Fr\'echet space of smooth symmetric $(2,0)$-tensor fields on $M$, with the usual $C^\infty$-topology.

Now let $\Riem^+ (M)^c \subset \Riem (M)^c$ be the subspace of all Riemannian metrics with positive scalar curvature (this is an open subspace). It follows from \cite[Theorem 13]{Palais} and \cite[Proposition A.11]{Hatcher} that $\Riem^+ (M)^c$ has the homotopy type of a CW complex.

If $h \in \Riem^+ (M)^{c}$ and $h=g+dt^2$ on $\partial M \times [0,c]$, then $\scal(g+dt^2 )=\scal(g)$ and so the metric $g$ on $\partial M$ necessarily has positive scalar curvature. This defines a restriction map
\[
 \res^{c}: \Riem^+ (M)^{c} \to \Riem^+ (\partial M),
\]
which is continuous. We define 
\[
 \Riem^+ (M)_g^{c} := (\res^{c})^{-1} (g),
\]
the space of all psc metrics on $M$ which on $\partial M \times [0,c]$ are equal to $g+dt^2$. Moreover, we define 
\[
\Riem^+ (M):= \colim_{c \to 0} \Riem^+ (M)^c
\]
and 
\[
\Riem^+ (M)_g:= \colim_{c \to 0} \Riem^+ (M)^c_g.
\]
If $b >c$, then $\Riem^+ (M)^b \subset \Riem^+ (M)^c$ and $\Riem^+ (M)_g^b \subset \Riem^+ (M)_g^c$, and it is elementary to see that the inclusion maps are homotopy equivalences \cite[Lemma 2.1]{BERW}. 

The restriction maps induce a restriction map
\[
  \res: \Riem^+ (M) \to \Riem^+ (\partial M)
\]
on the colimit, and there is a continuous bijection
\begin{equation}\label{continuous-bijections}
\Riem^+ (M)_g \to \res^{-1}(g). 
\end{equation}
There is no a priori reason why \eqref{continuous-bijections} should be a homeomorphism. However:
\begin{lem}\label{lem:cont-bijection-equivalence}
The map \eqref{continuous-bijections} is a weak homotopy equivalence. 
\end{lem}
\begin{proof}
The inclusion maps $\Riem^+ (M)^b \to \Riem^+ (M)^c$ and $\Riem^+ (M)_g^b \to \Riem^+ (M)_g^c$ are closed embeddings. Hence the Lemma then follows from the next one, which is a general fact.
%It follows that the system $(\Riem^+ (M)^{c})_{c}$ is strongly filtered \cite[Definition 3.4 and Lemma 3.6]{Strick}, in other words, if $f:K\to \Riem^+ (M)$ is a continuous map from a compact Hausdorff space, there is $c>0$, so that $f$ factors as $K \to \Riem^+ (M)^{c} \to \Riem^+ (M)$. The same conclusion holds for $(\Riem^+ (M)_g^{c})_{c}$. 
\end{proof}

\begin{lem}\label{two-topologies-on-fibnr4es}
Let $X_0 \to X_1 \to X_2 \to X_3 \to \ldots$ be a sequence of closed embeddings of Hausdorff spaces and let $f_n:X_n \to Y$ be a compatible sequence of maps. Then the continuous bijection $\psi:\colim_n (f_n^{-1}(y)) \to (\colim_n f_n)^{-1}(y)$ is a weak homotopy equivalence.
\end{lem}

\begin{proof}
It is enough to prove that if $K$ is compact Hausdorff and $h:K \to \colim_n (f_n^{-1}(y))$ is a map (of sets), then $h$ is continuous if and only if $\psi \circ h$ is continuous. If $g:= \psi \circ h$ is continuous, then we can consider $g$ as a map to $\colim_n X_n$. By \cite[Lemma 3.6]{Strick}, there is $n$ and $k:K \to X_n$ (continuous), so that $g:= i_n \circ k$ ($i_n:X_n \to \colim_n X_n$ is the natural map). Now $k$ maps into $f_n^{-1}(y)$, and so $h$ can be written as the composition $K \stackrel{k}{\to} f_n^{-1} (y) \to \colim_n (f_n^{-1}(y))$ of continuous maps.
\end{proof}

\subsection{The trace construction}

For the proof of both, Theorem \ref{thm:main} and Theorem \ref{thm:improved-chernysh-theorem}, we need a tedious but straightforward calculation using the standard formulas of Riemannian geometry. We include the proof because we do not know an explicit reference.

\begin{lem}\label{lem:scalar-curvature-of-trace}
Let $g: \bR \to \Riem (M)$ be a smooth path. Let $h:= dt^2+ g(t)$ be the induced metric on $\bR \times M$. Then the scalar curvature of $h$ is given by the formula
\[
\scal(h) = \scal(g(t)) +  \frac{3}{4}    g^{ik} g^{jl} g_{ij,0} g_{kl,0} - g^{kl} g_{kl,00} - \frac{1}{4} g^{ik} g^{jl}  g_{ik,0} g_{jl,0}.
\]
Here we use a local coordinate system in $M$ and the Einstein summation convention. Moreover, $g_{ij}$ are the components of the metric tensor of $g$, $g^{ij}$ the components of its inverse. A symbol as $g_{ij,k}$ denotes the derivative of $g_{ij}$ with respect to the $k$th coordinate and similarly for higher derivatives. The $0$th direction is the $\bR$-direction. %Of course, the same formula is true if $\bR$ is replaced by an interval.
\end{lem}
\begin{proof}
Fix a local coordinate system $(x_1, \ldots, x_d)$ on $M$ and define a coordinate system on $\bR \times M$ by taking $x_0 =t$, the $\bR$-variable. We now let $g_{ij}$ be the components of $g$ in these coordinates and $h_{ij}$ those of $h$. Let $g^{ij}$ and $h^{ij}$ be the components of the inverses of the metric tensors. Note that
\[
 h_{ij}=
 \begin{cases}
g_{ij} & i,j \geq 1,\\
1 & i=j=0,\\
0 & \text{otherwise}
 \end{cases}
\]
and 
\[
 h^{ij}=
 \begin{cases}
g^{ij} & i,j \geq 1,\\
1 & i=j=0,\\
0 & \text{otherwise}.
 \end{cases}
\]
We write $\Gamma_{ij}^k$, $R^{i}_{jkl}$ and $S$ for the Christoffel symbols, the components of the curvature tensor and the scalar curvature of $h$ and use the symbols $\gamma_{ij}^k$, $r^{i}_{jkl}$ and $s$ for those associated with $g$. %Moreover, a symbol such as $h-{ij,k}$ will denote the derivative of $h-{ij}$ with respect to the $k$th coordinate vector field, and a similar convention is used for higher derivatives. 
Without any further comment, we use the Einstein summation convention. With these notations in place, we have, essentially by definition, 
\[
\Gamma_{ij}^k = \frac{1}{2}  h^{kl} (h_{il,j}+ h_{jl,i} - h_{ij,l}) 
\]
\[
 R^k_{lij} = \Gamma_{jl,i}^k - \Gamma_{il,j}^k + \Gamma_{im}^k \Gamma_{jl}^m - \Gamma_{jm}^k \Gamma_{il}^m
\]
%\[
% R_{ijkl} = h_{im}R_{jkl}^m
%\]
\[
 S = h^{jl} R_{jkl}^k,
\]
see \cite[Corollary 3.3.1, (3.1.31), (3.3.6) and (3.3.19)]{Jost}. The symmetry property
\[
 \Gamma_{ij}^k = \Gamma_{ji}^k,
\]
is obvious. The same formulas of course hold for $g$ and its associated objects. Using these formulas, one computes
\[
 \Gamma_{ij}^k = 
 \begin{cases}
\gamma_{ij}^k & i,j,k \neq 0,\\
- \frac{1}{2}  g_{ij,0} & k=0, \; i,j \neq 0,\\
\frac{1}{2} g^{kl} (g_{jl,0}) & k,j \neq 0, \; i =0,\\
\frac{1}{2} g^{kl} (g_{il,0}) & k,i \neq 0, \; j =0,\\
0 & \text{otherwise}.
 \end{cases}
\]
The relevant components of the curvature tensor are
\[
 R^k_{0k0} =+ \frac{1}{4} g^{ki} g^{jl} g_{ij,0} g_{kl,0} - \frac{1}{2} g^{kl} g_{kl,00},
\]
\[
 R^0_{j0l} = - \frac{1}{2} g_{jl,00}+ \frac{1}{4} g^{mi} g_{lm,0} g_{ij,0}
\]
and (for $j,k,l \neq 0$)
\[
 R^k_{jkl} = r_{jkl}^k + \frac{1}{4} g^{ki} g_{il,0} g_{kj,0} - \frac{1}{4} g^{ki} g_{ik,0} g_{lj,0}.
\]

Altogether, we obtain (making an exception of the rule that we use the summation convention)
\[
 S= \sum_{k=0}^d R^k_{0k0} + \sum_{j,l \neq 0} g^{jl} R_{j0l}^0 + \sum_{j,k,l \neq 0}g^{jl} R^k_{jkl} =
\]
(switching back to the summation convention)
\[
 \frac{1}{4} g^{ki} g^{jl} g_{ij,0} g_{kl,0} - \frac{1}{2} g^{kl} g_{kl,00} - \frac{1}{2}g^{jl}  g_{jl,00} + \frac{1}{4}g^{jl}  g^{mi} g_{lm,0} g_{ij,0} + 
s + \frac{1}{4} g^{jl} g^{ki} g_{il,0} g_{kj,0} - \frac{1}{4} g^{jl} g^{ki} g_{ik,0} g_{lj,0} =
\]
 \[
s   + \frac{3}{4}    g^{ki} g^{jl} g_{ij,0} g_{kl,0} - g^{kl} g_{kl,00} - \frac{1}{4} g^{jl} g^{ki} g_{ik,0} g_{lj,0}.
\]
\end{proof}

\begin{lem}\label{gajer-lemma}\cite{Gajer}
Let $M$ be a compact manifold, $P$ a compact space and let $G: P \times [0,1] \to \Riem^+ (M)$ be continuous. Assume that for each $p \in P$, there is $B_p \in \bR$ such that $\scal (G(p,t)) \geq B_p$ for all $t \in [0,1]$. Then for each $\eta >0$, there is $\Lambda > 0$, such that if $f: \bR \to [0,1]$ is a smooth function with $\norm{f'}_{C^0}, \norm{f''}_{C^0} \leq \Lambda$, then the metric $G(p,f(t)) + dt^2$ on $M \times \bR$ satisfies $\scal (G(p,f(t)) + dt^2)\geq B_p - \eta$.
\end{lem}

\begin{proof}
Lemma \ref{lem:scalar-curvature-of-trace} shows that there is $C>0$ so that
\[
 \scal (G(p,f(t)+dt^2) \geq B_g - C(\norm{f'}_{C^0} + \norm{f''}_{C^0}),
\]
which immediately implies the claim.
\end{proof}

\subsection{Rotationally invariant metrics}

Let $\psi\colon (0,\infty)\times S^{k-1} \to \bR^k \setminus\{0\}, (t,v)\mapsto tv$ be the polar coordinate map. We denote by $d\xi^2$ the round metric on $S^{k-1}$. Furthermore, $S^{k-1}_r \subset \bR^k$ denotes the sphere of radius $r$. 

\begin{lem}\label{lem:preparation}
Let $g$ be an $O(k)$-invariant Riemannian metric on $B_R^k$, i.e. for all $A \in O(k)$, we have $A^*g=g$. 
\begin{enumerate}
\item There exist smooth functions $a,f\colon(0,R)\to (0,\infty)$, such that $\psi^*g = a(t)^2dt^2 + f(t)^2 d\xi^2$.
\item $a(t)\equiv 1$ holds if and only if the rays $t\mapsto tv$ are unit speed geodesics for all $v\in S^{k-1}$. In this case we call $g$ a \emph{normalized rotationally symmetric metric}. 
\item Under the hypothesis of (2), $f$ is the restriction of an odd smooth function $f\colon \bR\to\bR$ with $f'(0)=1$. We call $f$ the \emph{warping function of $g$}. 
\item In that situation, the scalar curvature of $g$ is given by 
\begin{equation}\label{curvature-cormula}
\scal(g) = (k-1)\left((k-2)\frac{1-f'^2}{f^2} - 2\frac{f''}f\right).
\end{equation}
\end{enumerate}
\end{lem}

\begin{proof}
For part (1), one uses that for each $v\in S^{k-1}$, there is an $A\in O(k)$ such that $Av=v$ and $A|_{v^\perp}=-\id$. It follows that at each point $0\ne x\in B_R^k$, the spaces $\mathrm{span}\{x\}$ and $T_x(S^{k-1}_{\norm{x}}) $ are orthogonal with respect to $g$. Since $d\xi^2$ is, up to a constant multiple, the only $O(k)$-invariant metric on $S^{k-1}$, the claim follows.
Part (2) is clear. %		\item $1=\norm{\dot\gamma(t)}_g^2 = \norm{v}_g^2 = a(t)dt^2(v,v) = a(t)$.
Part (3) can be found in \cite[\S 3.4]{petersen}, and the computation for part (4) in \cite[p. 69]{petersen}. 
\end{proof}

We denote the scalar curvature of the metric $dt^2 + f(t)^2 d\xi^2$ by
\begin{equation}\label{shortnotation-sclar-of-arping}
\sigma (f):=\scal (dt^2 + f(t)^2 d\xi^2)= (k-1)\left((k-2)\frac{1-f'^2}{f^2} - 2\frac{f''}f\right).
\end{equation}

The function $f(t)=\sin(t)$ on $[0,\pi)$ gives a metric which is isometric to the usual round metric on $S^k$. It has $\sigma(f)= k (k-1)$. Let us now give the precise definition of the torpedo metrics. 

\begin{defn}\label{defn:torpedometrix}
A \emph{torpedo function of radius $\delta>0$} is a function $f: [0,\infty) \to \bR$ which is the restriction of a smooth odd function with $f'(0)=1$, such that
\begin{enumerate}
\item $0 \leq f' \leq 1$,
\item $f'' \leq 0$, 
\item there is $R>0$ so that $f \equiv \delta$ near $[R,\infty)$,
\item $\sigma(f) \geq \frac{1}{\delta^2}(k-1)(k-2)$.
\end{enumerate}
The metric $dt^2 + f(t)^2 d \xi^2$ on $\bR^k$ is called a \emph{torpedo metric of radius $\delta$}.
\end{defn}

Let us give a concrete construction of a torpedo function. Let $\epsilon>0$ be small and let $u\colon[0,\infty)\to \bR$ be a function satisfying
\begin{itemize}
\item $u\equiv \id$ on $[0,\frac{\pi}{2}-\epsilon]$,
\item $u\equiv \frac{\pi}{2}$ on $[\frac{\pi}{2} + \epsilon,\infty]$,
\item $u''\leq 0$ (together with the previous conditions, this implies $0 \leq u'\leq 1$).
\end{itemize}
We define $h_1(t):=\sin(u(t))$. %The \emph{torpedo metric of radius $1$} is the rotationally symmetric metric on $\bR^k$ given by $dt^2 + f_1(t)^2 d\xi^2$. 
By \eqref{curvature-cormula} we have
\[
\sigma(h_1)  = u'(t)^2 k(k-1)  + (1-u'(t)^2)\frac{(k-1)(k-2)}{\sin(u(t))^2} - 2(k-1)\frac{\cos(u(t))}{\sin(u(t))}u'' \geq 
\]
\[
\geq u'(t)^2 k(k-1)  + (1-u'(t)^2)(k-1)(k-2)  \geq (k-1)(k-2),
\]
so that $h_1$ is indeed a torpedo function of radius $1$ (with $R\geq  \frac{\pi}{2} + \epsilon$). 
For $\delta>0$, the function
\begin{equation}\label{defn:warping-for-deltatorpedo}
h_\delta(t):=\delta h_1(\frac t\delta)
\end{equation}
is a torpedo function of radius $\delta$.%; note that 
%\[
%\sigma (h_\delta)(t)= \frac{1}{\delta^2} \sigma (h)(\frac{t}{\delta}) \geq \frac{1}{\delta^2} (k-1)(k-2). 
%\] 
\emph{For the rest of this paper, we fix a torpedo function $h_1$ of radius $1$, and define $h_\delta (t):= \delta h_1(\frac t\delta)$.}

\section{The parametrized Gromov--Lawson construction}\label{sec:glparam}

In this and the following section, we prove Theorem \ref{thm:main}, and we begin with the precise statement.
Let $N$ and $M$ be compact manifolds with collared boundary and let $\varphi: N \times \bR^k \to M$ be an open embedding with $k \geq 3$. We assume that $\varphi^{-1} (\partial M)= (\partial N) \times \bR^k$ and let $\partial \varphi: \partial N \times \bR^k \to \partial M$ be the induced embedding. Furthermore, we assume that $\phi$ is compatible with the chosen collars, that is, if $(x,t,v) \in (\partial N)\times [0, 1) \times \bR^k \subset N \times \bR^k$, then 
\[
\varphi(x,t,v) = (\partial \varphi (x,v),t) \in \partial M \times [0,1) \subset M.
\]
From now on, we usually identify $N \times \bR^k$ with an open subset of $M$ via $\varphi$.

Let $g_N$ be a Riemannian metric on $N$ which is of the form $g_{\partial N}+dt^2$ on $\partial N\times [0,1)$. \emph{It is not required that $\scal (g_N)>0$}. Let 
\[
A:= \inf (\scal(g_N))\in \bR
\]
and pick $\delta >0$ so that
\[
\frac{1}{\delta^2}(k-1)(k-2)+A >0.
\]
Let $g_\torp^k$ be a torpedo metric on $\bR^k$ of radius $\delta$, and let $R>0$ be as in Definition \ref{defn:torpedometrix}. For $c>0$, define
\[
 \Riem^+ (M,\varphi)^c := \{ g \in \Riem^+ (M)^c \vert  g |_{N \times B_R^k}= (g_N + g_\torp^k)|_{N \times B_R^k} \} 
\]
and $\Riem^+ (M,\varphi):= \colim_{c \to 0} \Riem^+ (M,\varphi)$.
\begin{thm}\label{thm:main-part1-precise}
The inclusion maps 
\[
 \Riem^+ (M,\varphi)^c \to  \Riem^+ (M)^c \; \text{and} \;  \Riem^+ (M,\varphi) \to  \Riem^+ (M)
\]
are weak homotopy equivalences.
\end{thm}
The proof that we give will apply simultaneously to both cases, and for notational simplicity, we deal only with $\Riem^+ (M,\varphi) \to  \Riem^+ (M)$. 
The proof is in two steps. We introduce an intermediate space $ \Riem^+ (M,\varphi) \subset \Riem^+_{rot} (M) \subset  \Riem^+ (M) $, which is defined to be 
\[
 \Riem^+_{rot} (M):= \{ g\in\Riem^+(M)\vert  g|_{N\times B_{R}^k} = g_N+g_{B_R^k}, g_{B_R^k} \text{rotationally symmetric, normalized}, \scal (g_{B_R^k}) >0\}.
\]
In this section, we show:

\begin{prop}\label{prop:from-all-to-rotationallysymmetric}
The inclusion map 
\[
 \Riem_{rot}^+ (M) \to \Riem^+ (M)
\]
is a weak homotopy equivalence.
\end{prop}

The proof of Proposition \ref{prop:from-all-to-rotationallysymmetric} is essentially the same as the original argument by Gromov and Lawson \cite{GroLaw} (but note that Rosenberg--Stolz \cite{RosSto} corrected mistakes in \cite{GroLaw}). 

%\subsection{Preliminaries for the proof of Proposition \ref{prop:from-all-to-rotationallysymmetric}}
\subsection{Adapting tubular neighborhoods}
In the proof of Theorem \ref{thm:main}, we shall use several devices to change a Riemannian metric. One such device (which plays a minor, more technical role) is by suitable isotopies. 

\begin{defn}
A Riemannian metric $g$ on $M$ is \emph{normalized on the $r_0$-tube around $N$} if for each $p \in N$ and $v \in S^{k-1}$, the curve $[0,r_0] \to N \times \bR^k \subset M$, $t \mapsto (p,tv)$, is a unit speed geodesic.
\end{defn}

For example, each $g \in \Riem^+_{rot}(M)$ is, by definition, normalized on the $R$-tube around $N$. 

\begin{prop}[Adapting tubular neighborhoods]\label{prop:adapting-tubularneighborhoods}
Let $(K,L)$ be a finite CW-pair and let $G: K \to \Riem (M)$ be continuous, so that $G(x)$ is normalized on the $r$-tube around $N$ when $x \in L$. Then there exists $r_0 \in (0,r]$ and a continuous map $F: [0,1] \times K \to \Diff(M)$ such that $ F(t,x)=\id$ if $(t,x) \in  (\{0\} \times K)\cup ([0,1] \times L)$, $F(t,x)|_N = \id$ and such that $F(1,x)^* G(x)$ is normalized on the $r_0$-tube around $N$, for all $x \in K$.
\end{prop}

\begin{proof}
The embedding $\phi:N \times \bR^k \to M$ identifies the normal bundle $\nu_N^M$ with the trivial vector bundle $N \times \bR^k$. For each Riemannian metric $g$ on $M$, there are maps 
\[
 \xymatrix{
\phi_g: N \times \bR^k \ar[r]^-{\cong} & \nu_N^M \ar[r]^{\gamma_g} & TN^\bot \ar@{..>}[r]^{\exp_g} & M.
 }
\]
The first is the fixed isomorphism, the second is induced by the bundle metric $g$ and the third is the Riemannian exponential map of $g$ (and is only partially defined). The metric $g$ is normalized on the $r_0$-tube around $N$ if and only if $\phi_g$ is defined on $N \times B_{r_0}^k$ and agrees with $\phi$ there. Since $N$ and $K$ are compact, there is $r_0>0$ so that $\phi_{G(x)} $ is defined on $N \times B^k_{r_0}$, injective and has image in $N \times \bR^k \subset M$. There is an isotopy 
\[
 H: [0,1] \times K \times (N \times B_{r_0}^k) \to M
\]
of embeddings such that $H(t,x,\_) = \phi_{G(x)}$ for all $(t,x) \in (\{0\}\times K )\cup ([0,1] \times L)$ and such that $H(1,x,\_)=\phi$ for all $x \in K$. In the case $K=*$, this follows from the well-known result that tubular neighborhoods are unique up to isotopy \cite[Theorem 4.5.3, 4.6.5]{Hirsch}. The proof given in loc.cit. carries over to the parametrized and relative case without change. 

An instance of the parametrized isotopy extension theorem \cite[Theorem 6.1.1]{WallDiffTop} shows that there exists $F:[0,1] \times K \to \Diff(M)$ with $F(t,x)=\id$ when $(t,x) \in (\{0\} \times K) \cup ([0,1] \times L)$ and $F(t,x)|_{N \times B^k_{r_0}} = H(t,x,\_)$. The Riemannian metric $F(1,t)^* G(x)$ is normalized on the $r_0$-tube around $N$. 
\end{proof}

\subsection{Gromov--Lawson curves}\label{sec:glcurves}

One important step in the proof of Theorem \ref{thm:main} (well-explained in e.g. \cite{RosSto}, \cite{Walsh11}) is to obtain a deformation of a psc metric $g$ on $M$ by a deformation of $M$ inside $M \times \bR$ and to take the metric induced by $g + dt^2$.

\begin{defn}\label{def:glcurve}
A \emph{Gromov--Lawson curve} $\Gamma$ is a smooth map $\Gamma\colon [0,1]\times [0,\infty)\to\bR^2$ such that
\begin{enumerate}
\item $\Gamma(0,s)= (0,s)$ for all $s$,
\item each curve $\Gamma_\lambda:= \Gamma (\lambda,\_)$ is an embedding,
\item there exist $\rho >r_0>0$ such that for all $\lambda \in [0,1]$, $\Gamma_\lambda (s)$ lies on the positive $r$-axis for all $s \geq r_0$ and $\Gamma_{\lambda} (s)= (0,s)$ for all $s \geq \rho$. We call $r_0$ the \emph{outer width of $\Gamma$}. 
%\item$\gamma$ is symmetric to the $x$-axis, i.e. $\gamma(-s)=R\circ\gamma(s)$ where $R$ is the reflection about the $x$-axis.
\item $\Gamma_1$ has a horizontal line segment of height $r_\infty$, i.e. there exist $0<y_4<y_5 \in\bR$ such that the line segment between the points $(y_4,r_\infty)$ and $(y_5,r_\infty)$ lies in the image of $\Gamma_1$. We call $r_\infty$ the \emph{inner width} of $\Gamma$, and $\ell:=|y_4-y_5|$ is the \emph{length of $\Gamma$}.
\item $\Gamma_\lambda(0)$ lies on the $y$-axis, and this is the only point where $\Gamma_\lambda$ meets the $y$-axis. Moreover, it does so at a right angle and follows the arc of a circle (of possibly infinite radius) in the region where $r \leq \frac{1}{2} r_\infty$.
\end{enumerate}
\end{defn}

A typical Gomov-Lawson curve is shown in figure \ref{figure:GL-curve}. The indicated points $(y_i,r_i)$ are important for the construction of these curves.%The division into $[s_i,s_{i+1}]$ arises from the construction in the proof of \ref{thm:glcurve}. 
\begin{figure}[h]\label{figure:GL-curve}
	\includegraphics[width=38em]{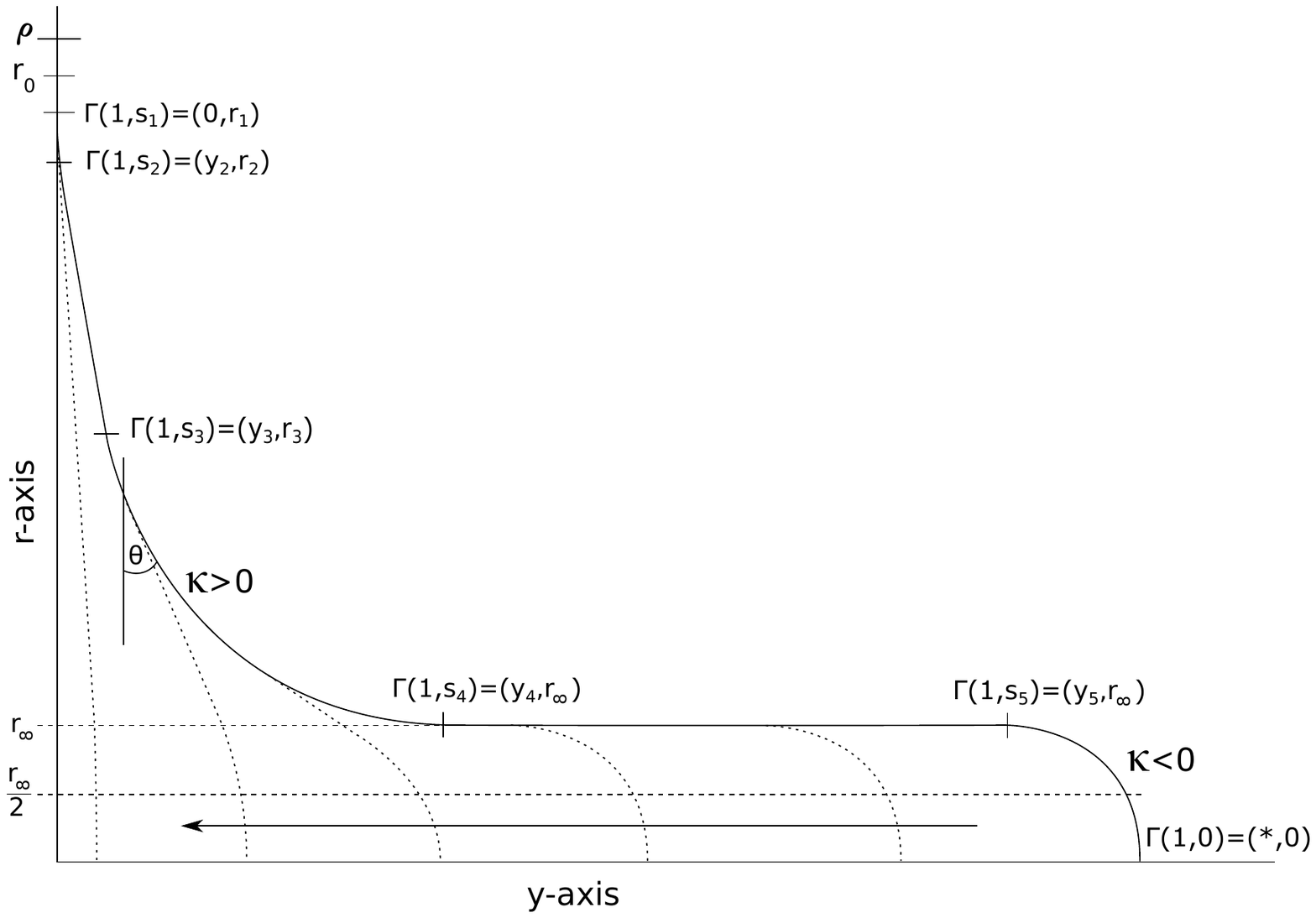}
\caption{A Gromov--Lawson curve $(\gamma,H)$}\label{fig:bendingcurve}
\end{figure}

%\begin{remark}
%	The names for $r_0$ and $r_\infty$ arise from the picture for $M_\gamma(g)$.
%\end{remark}

A Gromov--Lawson curve determines an isotopy of embeddings $E_\Gamma: [0,1] \times M \to M \times \bR$. Write $\Gamma^i_\lambda$, $i=1,2$, for the components of $\Gamma_\lambda$. First we define $E_\Gamma: [0,1] \times N \times \bR^k  \to N \times \bR^k  \times \bR$ by
\begin{equation}\label{eq:gromovlawsonembedding}
E_\Gamma (\lambda,p,v):= (p, \Gamma^2_\lambda (\norm{v}) \frac{v}{\norm{v}},\Gamma_\lambda^1 (\norm{v}))
\end{equation}
(which is smooth by the condition on $\Gamma$ near the $x$-axis). For $\norm{v} \geq \rho$, $E_\Gamma (\lambda,p,v)=(p,v)$, and so we can extend $E_\Gamma (\lambda,\_)$ as the identity over all of $M$. Note that $E_\Gamma (0,\_)$ is just the inclusion $x \mapsto (x,0)$. 
Let $g_{\Gamma_\lambda}$ be the Riemannian metric
\[
g_{\Gamma_\lambda}:= E_\Gamma (\lambda,\_)^* (g+dt^2)
\]
on $M$ obtained by restricting the product metric on $M \times \bR$ to the image of $E_\Gamma (\lambda,\_)$ and pulling back to $M$. 
The key argument for the proof of Proposition \ref{prop:from-all-to-rotationallysymmetric} is the following result.

\begin{prop}\label{prop:glcurve}
Let $K\subset\Riem(M)$ be compact. Suppose that each $g\in K$ is normalized on the $r_0$-tube around $N$ and that $\scal (g) \geq B_g$ for some $B_g$. For every $\epsilon_0>0$, $\eta >0$, there exists a Gromov--Lawson curve $\Gamma$ such that
\begin{enumerate}
\item the outer width of $\Gamma$ is at most $r_0$,
\item the inner width of $\Gamma$ is at most $\epsilon_0$,%and the horizontal line segment of $\gamma$ has positive length. %Furthermore the length of the horizontal line segment can be chosen arbitrarily long.
\item $\scal(g_{\Gamma_\lambda}) \geq B_g- \eta$ for all $g\in K$ and all $\lambda \in [0,1]$.
\item Moreover, for each $\ell>0$, we can arrange the length to be at least $\ell$.
\end{enumerate}
\end{prop}

\subsection{Construction of the Gromov--Lawson curve}

In this subsection, we prove Proposition \ref{prop:glcurve}. We need a formula for the scalar curvature of the metric $g_{\Gamma_\lambda}$. 

Let $I \subset \bR$ be an interval and let $\gamma: I \to \bR^2_+:= \{(y,r) \vert r \geq 0\}$ be a smooth embedded curve. We assume that whenever $\gamma(t)$ lies on the $y$-axis, then near $t$, $\gamma$ follows a circle of possibly infinite radius perpendicular to the $y$-axis. Consider the hypersurface
\[
Q_\gamma := \{(p,v,r) \in N \times \bR^k \times \bR \vert (r, \norm{v}) \in \im (\gamma)\} \subset N \times \bR^k \times \bR \subset M \times \bR
\]
(which is smooth because of the condition on $\gamma$ near the $y$-axis). Let us recall some formulas from the geometry of plane curves. In the situation we consider, the derivative vector of $\gamma$ will lie in the fourth quadrant. We let $\theta$ be the angle between $\gamma$ and the negative $r$-axis and let $\kappa$ be the signed curvature of $\gamma$. 
%At a given point $(s,r )=\gamma(s) \in \im (\gamma)$, we let $\theta = \theta(s) \in [0,\frac{\pi}{2}]$ be the angle between $\gamma$ and the vertical line through $(t,r)$, and we let $\kappa=\kappa (s)$ be the curvature of $\gamma$ at this point. 
If $\gamma$ is \emph{parametrized by arc-length}, the curvature is given by 
\begin{equation}\label{eq:curvature-of-curve}
\ddot\gamma(s)=\begin{pmatrix}0&-1\\1&0\end{pmatrix}\cdot \dot\gamma(s)\cdot \kappa(s)
\end{equation}
or 
\[
 \kappa (s) = \langle \ddot{\gamma}(s), \twomatrix{0}{-1}{1}{0} \dot{\gamma}(s)\rangle.
\]
The angle is given by the formula
\[
 \sin (\theta) = \langle \dot{\gamma},e_1\rangle; \;  \cos (\theta) = -\langle \dot{\gamma},e_2\rangle.
\]
Note that
\begin{equation}\label{eq:derivativd-angle}
\frac{d}{ds}\theta(s)= \kappa(s). 
\end{equation}
%\begin{defn}
%A Riemannian metric $g$ on $M$ is \emph{normalized on the $r_0$-tube around $N$} if for each $p \in N$ and $v \in S^{k-1}$, the curve $[0,r_0] \to N \times \bR^k \subset M$, $t \mapsto tv$, is a unit speed geodesic.
%\end{defn}
If $\gamma$ meets the $y$-axis in a circle of radius $\rho<\infty$, then $\kappa =-\frac{1}{\rho}$, $r=\sin (\theta) \rho$ near that point. 

For a given Riemannian metric $g$ on $N\times \bR^k$, we get the Riemannian metric $g_\gamma$ on $Q_\gamma$, obtained by restricting the product metric $g+dt^2$ on $N \times \bR^{k+1}$ to $Q_\gamma$. 
\begin{lem}\label{lem:formula}
Let $K\subset \Riem (M)$ be compact. Assume that all $g \in K$ are normalized on the $r_0$-tube around $N$. Then there exists $0<r_1 \leq r_0$ and $C>0$ such that for all $g \in K$, and for all immersed curves in the region $\{ (y,t) \in \bR^2 \vert 0 < y \leq r_1\}$, we have
\[
\scal(g_\gamma) \ge \scal(g) + |\kappa| \Bigl( -\sign(\kappa) \frac{2(k-1)\sin(\theta)}{r} -C\sin (\theta) \Bigr) + \frac{(k-1)(k-2)\sin^2(\theta)}{r^2} - C \frac{\sin(\theta)^2}{r} .
\]
\end{lem}

\begin{rem}
This estimate originates from the curvature formula computed in \cite{GroLaw}, \cite{RosSto}, \cite{Chernysh} or \cite{Walsh11}. These papers however contain a small computational error: There the formula has either $\kappa\sin(\theta)\frac{k-1}{r}$ instead of $\kappa\sin(\theta)\frac{2(k-1)}{r}$ or $2\frac{(k-1)(k-2)}{r^2}$ instead of $\frac{(k-1)(k-2)}{r^2}$. We will point out in the proof of Lemma \ref{lem:formula} where the error occurs and in Remark \ref{rem:impact-of-mistake-on-GL-construction} below, we discuss what impact this has on the proof of Proposition \ref{prop:from-all-to-rotationallysymmetric}.
%		\item The correct curvature formula is
%		\begin{align*}
%			\scal({M_\gamma}(g)) = \scal(g) &+ \sin^2\theta\cdot O(1) + \kappa\cdot\sin\theta\cdot \Bigl( -  \frac{2\cdot(k-1)}r + O(1) \Bigr)\\
%					& + \sin^2\theta\cdot\Bigl( \frac{(k-1)(k-2)}{r^2} + \frac{2\cdot(k-1)}r\cdot O(1) \Bigr)\\
%					&+ \sin\theta\cdot \kappa \cdot(k-1)\cdot O(r).
%		\end{align*} 	\end{enumerate}
\end{rem}

\begin{proof}[Proof of \ref{lem:formula}]
We will abbreviate this proof. A detailed computation can be found in the appendix of \cite{Walsh11}. The curvature of $g_\gamma$ is given by 
$$\scal(g_\gamma) = \overbrace{\scal(g+dt^2)}^{=\scal(g)} + 2\sum_{i<j}\lambda_i\lambda_j$$
where $\lambda_i$ are the principal curvatures of the hypersurface $Q_\gamma \subset  N\times\bR^k\times\bR$. These are given by 
$$
\lambda_j = \begin{cases}
\kappa &\text{ if } j = 1\\
\sin(\theta)\cdot\bigl(-\frac1r + O(r)\bigr) &\text{ if } 2\le j\le k\\
\sin(\theta)\cdot O(1) &\text{ if } k+1 \le j \le d.
\end{cases}
$$
where the $O$-terms arise from the coefficients of the Taylor expansion of $g$ and depend continuously on $g$. Therefore we get\footnote{This is where the error \cite{Walsh11} occurs: $\binom{k-1}2=\frac{(k-1)(k-2)}2$ is miscounted as $(k-1)(k-2)$.}:
\begin{align*}
		\sum_{i<j} \lambda_i\lambda_j &= \kappa\sum_{i=2}^k \lambda_i +  \kappa\sum_{i=k+1}^d \lambda_i + \sum_{2\le i<j\le k}\lambda_i\lambda_j + \sum_{2\le i\le k < j\le d}\lambda_i\lambda_j + \sum_{k< i<j\le d}\lambda_i\lambda_j \\
			&= \sin(\theta)\cdot \kappa\cdot(k-1)\cdot\bigl(-\frac1r + O(r)\bigr) + \kappa\cdot O(1)\cdot\sin(\theta)\\
				& + \underbrace{\frac{(k-1)(k-2)}2}_{=\binom{k-1}2}\cdot\underbrace{\bigl(-\frac1r + O(r)\bigr)^2}_{=\Bigl(\frac1{r^2} + O(1)\Bigr)}\cdot\sin(\theta)^2\\	
				& + \bigl(-\frac1r + O(r)\bigr)\cdot O(1)\cdot\sin(\theta)^2 + O(1)\cdot\sin(\theta)^2 .\\
\end{align*}
Rearranging all terms finishes the proof. 
\end{proof}

\begin{cor}\label{cor:curvature-on-little-spehres}
Let $K \subset \Riem (M)$ be compact. Assume that all $g \in K$ are normalized on the $r_0$-tube around $N$. Then for each $B \in \bR$, there exists $\eps_0 >0$ such that for $\eps \in (0,\eps_0)$, the restriction of $g$ to $N \times S^{k-1}_{\eps}$ has scalar curvature at least $B$. 
\end{cor}

\begin{proof}
Consider the curve $\gamma (s):= (s,\eps)$. The restriction of the product metric $g + dt^2$ to $Q_\gamma$ is the product of $g|_{N \times S^{k-1}_{\eps}}$ and $dt^2$. Hence $\scal (g_\gamma)= \scal (g|_{N \times S^{k-1}_{\eps}})$. In the case at hand, the curvature of $\gamma$ is $\kappa=0$, and $\theta = \frac{\pi}{2}$. Hence from Lemma \ref{lem:formula}, we get
\[
\scal(g_\gamma) \ge \scal(g)+ \frac{(k-1)(k-2)}{\eps^2} - \frac{C}{\eps}.
\]
If $\eps$ is small enough, the term $\frac{(k-1)(k-2)}{\eps^2}$ dominates all other terms. 
\end{proof}

%We will prove the existence of $\gamma$ in \ref{prop:glcurve}. After that we will construct the homotopy $H$ in \ref{prop:pullingback}.

%\begin{prop}\label{prop:glcurve}
%	For every $\epsilon>0$ there exists a curve $\gamma\colon\bR\to\bR^2$ satisfying properties (1) - (3) from \ref{Definition}{def:glcurve} and properties (1) and (3) from \ref{thm:glcurve}.
%\end{prop}

Another auxiliary result is needed for the proof of Proposition \ref{prop:glcurve}.

\begin{lem}\label{lem:diffeq}
Let $a>0$ and consider the ordinary differential equation 
\begin{equation}\label{diffeq1}
	h''= \frac{1+ {h'}^2}{a\cdot h}.
\end{equation}
For any choice of initial values $h(t_0)>0$ and $h'(t_0)<0$, there is $T>t_0$ and a solution $h: [t_0,T] \to (0,\infty)$ such that $h' \le 0$ and $h'(T) = 0$.
\end{lem}
\begin{proof}%[Proof of \ref{lem:diffeq}]
Let $h\colon [t_0,T_1)\to\bR$, $T_1\in(t_0,\infty]$ be a maximal solution. We do not want to decide whether $T_1=\infty$ or $T_1 <\infty$ and show that both cases lead to the desired conclusion. The quantity
$$C(t):=h^{\frac1a}(t)\cdot \frac{1}{\sqrt{1+h'(t)^2}}$$ 
is constant as can be seen by differentiating. Also $C(t)>0$ because of the initial conditions, and $h$ is bounded from below by $C(t_0)^a$. 

Suppose first $T_1 =\infty$. If $h'(t) <0$ for all $t \geq t_0$, then $h$ is decreasing and
\[
h''(t)=\frac{1+ {h^\prime(t)}^2}{a\cdot h(t)}\ge\frac1{a\cdot h(t)}\ge\frac1{a\cdot h(t_0)}=:b>0
\]
implies $h'(t) \geq b (t-t_0)+h'(t_0)$, which is a contradiction. Hence there is $T>t_0$ with $h'(T)=0$. 

If $T_1 <\infty$, we consider the trajectory of $(h(t),h'(t))$ in the phase diagram. Since $C(t)$ is constant, this trajectory lies on the level set $C^{-1}(C(t_0))$. Because $T_1<\infty$, this trajectory leaves every compact subset of $\bR^2$. The shape of the level set is so that this implies $\lim_{t \to T_1} h(t)=+\infty$. Hence by Rolle's theorem, $h'(T)=0$ for some (minimal) $T>t_0$.
\end{proof}

\begin{rem}
One can solve \eqref{diffeq1} explicitly, using that $C$ is conserved. The above proof seems more efficient to us, though.
\end{rem}

\begin{proof}[Proof of Proposition \ref{prop:glcurve}]
%We first construct the images of the curves $\gamma$ and $H(\lambda,\_)$ and obtain $\gamma$ by a suitable parametrization. 
We first construct a piecewise $C^2$ curve $\alpha$ and a homotopy $\alpha_\lambda$ of such curves. By a smoothing procedure, we obtain a homotopy $\beta_\lambda$ of smooth curves which will yield $\Gamma_\lambda$ by a suitable reparametrization. We begin with the curve $\alpha=\alpha_1$. Let us pick some constants first. 
%The construction begins with choices of numbers $r_0 \geq r_1 \geq r_2 >r_3 >r_4 >0$.
\begin{enumerate}
\item Choose $a >  \frac{2}{k-2}$ (note that $-\frac{2}{a} +k-2 >0$).
\item Choose $\rho>r_0$ arbitrarily. 
\item $0<r_1\leq r_0$ is chosen, so that the curvature estimate from Lemma \ref{lem:formula} is valid for $r\leq r_1$, with a constant $C>0$.
\item Next, we choose $0<r_2 \leq r_1$ so that
\[
 r_2 \leq \frac{k-1}{C}
\]
\item and pick $r_3 \in (0,r_2)$ arbitrarily.% (the choice of $r_4$ will be explained shortly).
\end{enumerate}
%First we pick $\rho>R$\footnote{$R$ is the radius where we want roatiobnal symmetry}. Then $r_0 \leq R$. Next,  
Let us explain the choice of $r_2$. 
\begin{claim}\label{claim-GLcurveproof}
If $\gamma$ is an immersed curve in the region $\{(y,r)\vert r \in (0,r_2]\}$ whose signed curvature $\kappa$ is nonpositive, then $\scal(g_\gamma) \geq B_g$. 
\end{claim}
To see this, estimate 
\begin{equation}\label{eq:proofglcurve1}
  \frac{(k-1)(k-2)\sin^2(\theta)}{r^2} - C \frac{\sin(\theta)^2}{r} = \frac{\sin(\theta)^2}{r} \Bigl( \frac{(k-1)(k-2)}{r} - C\Bigr) \geq \frac{\sin(\theta)^2}{r} (k-3)C \geq 0,
\end{equation}
%Hence if $\kappa=0$, then $\scal(g_\gamma) >0$ by \ref{lem:formula}; in other words: if $\gamma$ is a straight-line curve in the region $r \in [0,r_2]$, then $g_\gamma$ has positive scalar curvature. 
using $r \leq r_2$. If $\kappa \leq 0$, then 
\begin{equation}\label{eq:proofglcurve2}
|\kappa| \Bigl( -\sign(\kappa) \frac{2(k-1)\sin(\theta)}{r} -C\sin (\theta) \Bigr)= \sin(\theta) |\kappa| \Bigl( \frac{2(k-1)}{r}-C \Bigr)  \geq 0.
%  -  \frac{2(k-1)\kappa\cdot\sin(\theta)}{r} - C |\kappa| \sin (\theta) = \sin(\theta) |\kappa| \Bigl( \frac{2(k-1)}{r}-C \Bigr)  \geq 0.
\end{equation}
Together with Lemma \ref{lem:formula}, these two inequalities establish Claim \ref{claim-GLcurveproof}.

Let us now construct the first part of $\alpha$. One device to construct a (unit speed) curve is by prescribing its curvature function. More precisely, let $J \subset \bR$ be an interval and $s_0 \in J$. If a function $\kappa:J \to \bR$ and initial values $\gamma (s_0)$ and $\dot{\gamma}(s_0)$ (the latter of unit length) are given, then the solution to the differential equation 
\[
 \ddot{\gamma}(s) =\kappa (s) \twomatrix{}{-1}{1}{} \dot{\gamma}(s)
\]
is a unit speed curve with curvature function $\kappa$. If $\kappa$ is piecewise continuous, then $\gamma$ is piecewise $C^2$. We write $\theta(s)$ for the angle of the curve $\gamma(s)$.% Recall that
%\[
% \frac{d}{ds} \theta(s)= \kappa(s).
%\]
\begin{enumerate}
%\item The first part of $\tilde{\gamma}$ coincides with the $y$-axis above the point $r_2$. 
%\item $\alpha_1$ begins at the point $(0,r_2)$. 
\item Choose $0<\delta <\frac{r_2-r_3}{3}$. Consider the function\footnote{$\chi_S: X \to \{0,1\}$ denotes the characteristic function of a subset $S \subset X$.} $\kappa(s) = q  \chi_{[\delta,2\delta]}(s)$, for some $q>0$, and the unit speed curve $\alpha$ on $[0,\infty)$ with initial values $\alpha(0)=(y_2,r_2):=(0,r_2)$ and $\dot\alpha(0)=(0,-1)$ and curvature function $\kappa$. If $q \delta< \frac{\pi}{2}$, the angle of the curve $\alpha$ will always be less than $\frac{\pi}{2}$, and so it crosses the horizontal line of height $r_3$ in some point $(y_3,r_3)$, with an angle
\[
\theta_0 = q \delta < \frac{\pi}{2}.
\]
If $q$ satisfies 
\[
 q \Bigl( \frac{2(k-1)}{r_3} + C \Bigr) \leq \frac{1}{2}\eta,
\]
we claim that $\scal(g_\alpha) \geq B_g-\frac{1}{2}\eta$. This follows from \eqref{eq:proofglcurve1}, Lemma \ref{lem:formula}, and the estimate 
\begin{equation}\label{proof-gl-curv-estimate1}
  -  \frac{2(k-1)\kappa\cdot\sin(\theta)}{r}  - C|\kappa| \sin (\theta)  \geq - \frac{2(k-1) q}{r_3} -Cq \geq - \frac{1}{2}\eta.
\end{equation}
%which together imply 
%\begin{equation}\label{proof-gl-curv-estimate1}
%  \scal(g_{\tilde{\gamma}}) \ge B -  \frac{2(k-1)\kappa\cdot\sin(\theta)}{r}  - C|\kappa| \sin (\theta) \geq B- \frac{2(k-1) q}{r_3} -Cq \geq \frac{1}{2}B.
%\end{equation}
%The curve $\tilde{\gamma}$ crosses the horizontal line of height $r_3$ as a straight line (of some angle $\theta_0>0$).
%Between the height $r_2$ and $r_3$, the curve is bent upwards, until it has an angle $\theta_0 >0$. The only thing that is important is that this angle $\theta_0$ is positive. 
%The best way to describe this bend is to parametrize $\gamma$ by arclength and to give it by its curvature function, using the formula \eqref{eq:curvature-of-curve}. The initial values are $\gamma (0):= (0,r_2)$ and $\gamma'(0)= (0,-1)$. Now let $u$ be a bump function supported in a small interval $(0,\delta)$ and let $q>0$. The curvature function of $\gamma$ is $q u(t)$. We have to pick $q>0$ so small that the positivity of the scalar curvature is preserved. This is the case if 
%\[
% q \Bigl( \frac{k-1}{r_3}+C\Bigr) \leq \frac{1}{ 2} B. 
%\]
\item Now we pick $r_4 >0$ so that $r_4 \leq \epsilon_0$, $r_4<r_3$ and 
\[
 r_4 \leq \frac{(k-1) \sin(\theta_0)^2 (-\frac{2}{a}+k-2)}{C(1+\frac{1}{a})}.
\]
Between height $r_3$ and $r_4$, the curve $\alpha$ follows the straight line of slope $\theta_0$ (there is no problem with the psc condition, by Claim \ref{claim-GLcurveproof}). It crosses the horizontal line of height $r_4$ at a certain point $(y_4,r_4)$. If after that point, the curve $\alpha$ satisfies
\begin{equation}\label{eq:conditions-for-critialbend}
 \theta_0 \leq \theta \leq \frac{\pi}{2}, \; 0 \leq \kappa \leq \frac{\sin (\theta)}{ar},\;0< r\leq r_4,
\end{equation}
we estimate, using that $a >  \frac{2}{k-2}$,
\[
 |\kappa| \Bigl( - \sign(\kappa) \frac{2(k-1) \sin(\theta)}{r}- C \sin(\theta) \Bigr) + \Bigl( \frac{(k-1)(k-2) \sin(\theta)^2}{r^2}- C \frac{\sin(\theta)^2}{r} \Bigr)\geq 
\]
\[
\frac{\sin(\theta)}{ar} \Bigl( -  \frac{2(k-1) \sin(\theta)}{r}- C \sin(\theta) \Bigr) + \Bigl( \frac{(k-1)(k-2) \sin(\theta)^2}{r^2}- C \frac{\sin(\theta)^2}{r} \Bigr) = 
\]
\[
 = \frac{\sin(\theta)^2 (k-1)}{r^2} \Bigl( -\frac{2}{a} + k-2  \Bigr) - \frac{C \sin(\theta)^2}{r}\Bigl( 1 + \frac{1}{a} \Bigr) \geq 
\]
\[
\geq  \frac{\sin(\theta_0)^2 (k-1)}{r r_4} \Bigl( -\frac{2}{a} + k-2  \Bigr) -  \frac{C }{r} \Bigl( 1 + \frac{1}{a} \Bigr) = 
\]
\[
 \frac {1}{r} \Bigl(  \frac{\sin(\theta_0)^2 (k-1) (-\frac{2}{a}+k-2)}{r_4} - C(1+\frac{1}{a})  \Bigr) \geq 0
\]
by the definition of $r_4$.
%\[
% \geq \frac{1}{r} \Bigl( C(1+\frac{1}{a}) -C(1+\frac{1}{a}) \Bigr) =0.
%\]
Altogether, Lemma \ref{lem:formula} shows that \eqref{eq:conditions-for-critialbend} implies
\begin{equation}\label{proof-gl-curv-estimate2}
\scal(g_\gamma) \geq B_g
\end{equation}
in this region.

We construct the curve $\alpha$ satisfying \eqref{eq:conditions-for-critialbend} as the graph of a function $f:[y_4,y_5]\to \bR$. For curves of the form $t\mapsto (t,f(t))$, we have 
\[
\kappa=\frac{f''}{(\sqrt{1+f'^2})^3}; \; \sin(\theta)=\frac{1}{\sqrt{1+f'^2}} ,
\]
see e.g. \cite[p. 41]{Baer}. Hence if we take $f$ as the solution of the ordinary differential equation
\[
 f'' = \frac{1+{f'}^2}{a f}
\]
with initial values 
\[
f(y_4)=r_4; \;  f' (y_4)= -\frac{\cos(\theta_0)}{\sin(\theta_0)}<0,
\]
then the curve $\alpha (t)=(t,f(t))$ satisfies \eqref{eq:conditions-for-critialbend}. By Lemma \ref{lem:diffeq}, there is a solution $f: [y_4,y_5] \to (0,\infty)$ with $f' \leq 0$ and $f'(y_5)=0$. Let $r_5 := f(y_5)>0$, and we let $\alpha$ be the graph of $f$ in this region. 
\item From the point $(y_5,r_5)$ on, the curve $\alpha$ follows a straight horizontal line, of length $2 \ell$ (a little more than $\ell$ would suffice), until it reaches the point $(y_6,r_6):=(y_5+2 \ell,r_5)$. Since $\kappa \equiv 0$, there is no problem with the psc condition here, by Claim \ref{claim-GLcurveproof}. We let $r_\infty:=r_5=r_6$. The last piece of the curve (until it hits the $y$-axis) will be constructed at the end of the proof.
%\item The last piece of $\alpha$ is a circle of radius $r_5$, until it hits the $x$-axis in the point $(t_7,r_7)=(t_7,0)$. Since the curvature of $\alpha$ is negative, there is no problem with the psc condition here as well (because we are below the line with height $r_2$).
\end{enumerate}
Now we parametrize the curve $\alpha$ by arclength, beginning at the point $\alpha(s_2)=(0,r_2)$ and call the reparametrized curve also $\alpha$. Let $s_6>s_5>s_4>s_3>s_2$ be the points with $\alpha(s_i)=(y_i,r_i)$. The curve $\alpha$ is entirely determined by its curvature function $\kappa$. The function $\kappa$ is zero outside the intervals $[s_0+\delta,s_0+2 \delta]$ and $[r_4,r_5]$. We have
\[
 \kappa|_{[s_0+\delta,s_0+2 \delta]}=q>0; \; \kappa_{[r_4,r_5]} >0.%; \; \kappa|_{[r_6,r_7]}= - \frac{\pi}{2} \frac{1}{r_7-r_6} .
\]
%By construction and by \eqref{eq:derivativd-angle}, we have
Now we pick $0 < \omega \ll \min (\frac{r_5}{2},\delta,\ell)$. Let us now construct a homotopy of piecewise $C^2$-curves, which are defined on intervals of varying length $[s_2, s_6(\lambda)]$. Let $s_6(1):= s_6$ and $\alpha_1:= \alpha$. During the homotopy interval $[\frac{1}{2},1]$, we shrink down the size of the horizontal piece until it is $\omega$ (so that $s_6(\frac{1}{2})= s_5+\omega$), and $s_i (\lambda)= s_i$, for $i=2,3,4,5$, $\lambda \in [\frac12,1]$. 

For $\lambda \in [0,\frac12]$, consider the curvature function $\kappa_\lambda: = \chi_{[s_2, s_2 + 2\lambda (s_5-s_2)]}\kappa$ and let $s_5(\lambda)$ be the point where the curve $\alpha_\lambda$ with curvature function $\kappa_\lambda$ reaches the horizontal line of height $r_5$ (it is always the case that $s_5 (\lambda)\geq 2\lambda (s_5-s_2)]$) and put $s_6 (\lambda):= s_5 (\lambda) + \omega$.% and after that point, we follow the circle of appropriate radius (constant curvature function). 

By construction, the curves $\alpha_\lambda$ satisfy the psc condition $\scal (g_{\alpha_\lambda})\geq B_g-\frac{\eta}{2}$, $\alpha_0$ is the straight line on the $r$-axis, and $\alpha_1=\alpha$. The curves $\alpha_\lambda$ are $C^1$ and piecewise $C^2$, and we need to smoothen them. 

To that end, pick an even, smooth, nonnegative bump function $\xi$ with support in $(-\frac{1}{4},\frac{1}{4})$ and integral $1$ and let $\xi_{u}(t):= \frac{1}{u} \xi(u t)$. For $u \in (0,\omega]$, we define the smooth curve 
\[
\beta_{\lambda,u}:= \xi_u \ast \alpha_\lambda
\]
using convolution. If $u\leq  \omega$, then $\beta_{\lambda,u}\equiv \beta_\lambda$ near $s_6(\lambda)$ and near $s_2$, and $\beta_{0,u}$ lies on the $r$-axis. This holds because $\xi_u \ast f (t)= f(t)$ if $f$ is linear near $t$. 
%the qualitative properties of the curves remain the same (for example $\beta_{\lambda,u}$ is a straight line near $s_6(\lambda)$, with the same angle as $\alpha_\lambda$). The reason for this is that (a)
%\[
% \int_{s_0}^{s_5} \kappa(s) ds = \frac{\pi}{2},%; \; \int_{r_6}^{r_7} \kappa(s) ds = -\frac{\pi}{2}.
%\]
%(b) on each interval of length $\omega$, the function $\kappa_\lambda$ is either nonnegative or nonpositive and (c) that for two $L^1$-function $v,w \geq 0$ on $\bR$, the convolution $v \ast w$ is nonnegative and $\norm{v\ast w}_{L^1} = \norm{v}_{L^1} \norm{w}_{L^1}$. 

For small enough $u$, the curve $\beta_{\lambda,u}$ satisfies the positive scalar curvature condition, namely $\scal(g_{\beta_{\lambda,u}})\geq B_g-\eta$. This is no issue at point near which $\kappa_\lambda$ is continuous. Near the discontinuity points, the angle and height of $\beta_{\lambda,u}$ is close to that for $\alpha_\lambda$, while the curvature of $\beta_{\lambda,u}$ oscillates between the minimum and maximum value of $\kappa_\lambda$. The decisive estimates \eqref{eq:proofglcurve1}, \eqref{proof-gl-curv-estimate1} and \eqref{proof-gl-curv-estimate2} all hold if $\kappa$ lies between $0$ and the allowed maximum value. Note, however, that we might loose a bit scalar curvature. 

To construct the last piece of the curves $\beta_{\lambda,u}$ on an interval $[s_6(\lambda),s_7(\lambda)]$, we take a smooth family of curves $\gamma_\lambda: [s_6(\lambda),s_7(\lambda)] \to \bR^2$, such that
\begin{itemize}
\item $\gamma_\lambda$ begins at the point $\beta_{\lambda,u} (s_6(\lambda))$, as a straight line with the same angle as $\beta_{\lambda,u}$,
\item $\gamma_{\lambda} (s_7(\lambda))$ lies on the $y$-axis, and except on the interval $[s_6(\lambda),s_6(\lambda)+\omega]$, it is a circle,
\item the curvature of $\gamma_\lambda$ is $\leq 0$.
\end{itemize}
These conditions enforce that $\gamma_0$ lies on the $r$-axis. The construction of such curves is easy and left to the reader. By Claim \ref{claim-GLcurveproof}, there is no problem with the psc condition. 

Finally, the curve $\Gamma_\lambda$ is obtained by reparametrization (of the form $\Gamma_\lambda(s):= \beta_{\lambda,u}(s_7(\lambda)-s)$). It is extended to all of $[0,\infty)$, so that above $\rho$, it is just the curve $s \mapsto (0,s)$.
\end{proof}

\begin{rem}\label{rem:impact-of-mistake-on-GL-construction}
The above proof is almost the same as that of the corresponding result in \cite{RosSto} or \cite{Walsh11}. The difference is that in loc.cit., the slightly incorrect version of the curvature formula \eqref{lem:formula} is used. This allows the choice $a=2$ in the quoted papers. In that case, the differential equation \eqref{diffeq1} has a simple explicit solution. We can pick $a=2$ if $k>4$, but if $k=3$, we need $a>2$, and the argument in loc.cit. does not work as stated there.
\end{rem}

\subsection{Completion of the proof of Proposition \ref{prop:from-all-to-rotationallysymmetric}}

We now give the proof of Proposition \ref{prop:from-all-to-rotationallysymmetric}. We shall use the following well-known criterion for a map to be a weak equivalence.
\begin{prop}\label{prop:heq}
Let $j:X \to Y$ be the inclusion of a subspace. Then the following are equivalent:
\begin{enumerate}
\item $j$ is a weak homotopy equivalence,
%\item For every $n \geq 0$ and every map $G_0\colon D^n\to Y$ such that $G(S^{n-1})\subset X$ there exists a homotopy $G_\lambda$ such that $G_1(D^n)\subset X$ and $G_\lambda|_{S^{n-1}}=G_0|_{S^{n-1}}$  for all $\lambda  \in [0,1]$.
\item for every $n \geq 0$ and every map $G_0\colon D^n\to Y$ such that $G(S^{n-1})\subset X$, there exists a homotopy $G_\lambda$ starting with $G_0$ such that $G_1(D^n)\subset X$ and $G_\lambda(S^{n-1}) \subset X$  for all $\lambda \in [0,1]$.
%\item For every $n$ and every map $G\colon D^n\to Y$ such that $G(S^{n-1})\subset X$ there exists a homotopy $G_\lambda$ such that $G_0=G$, $G_1(D^n)\subset X$ and $G_\lambda(S^{n-1})\subset X$ for all $\lambda$.
\end{enumerate}
\end{prop}

%\begin{proof}
%$(2) \Rightarrow (1)$: We show that $G_0$ is homotopic relative to $S^{n-1}$ to a map $G_1'$ into $X$ and invoke a standard lemma (e.g. \cite[p. 136]{Gray}). Let $G_0: D^n \to Y$ be a map with $G_0 (S^{n-1}) \subset X$. The map $G_0$ is homotopic, relative to $S^{n-1}$, to the map 
%\[
%G'_{\frac{1}{2}} (x):=
%\begin{cases}
%G_0 (\frac{x}{\norm{x}}) & \norm{x} \geq \frac{1}{2},\\
%G_0 (2 x) & \norm{x} \leq \frac{1}{2}.
%\end{cases}
%\]
%For $\lambda \geq \frac{1}{2}$, define
%\[
%G'_\lambda (x):=
%\begin{cases}
%G_{(2-2\norm{x})(2\lambda-1)}\left(\frac{x}{\norm{x}}\right) &  \norm{x}\ge\frac{1}{2},\\
%G_{2\lambda-1} (2x) & \norm{x} \leq \frac{1}{2}.
%\end{cases}
%\]
%The other implication is likewise easy and not important for us.
%\end{proof}

So we let 
\begin{equation}\label{proof-firstpart-diag1}
 \xymatrix{
 S^{n-1} \ar[d] \ar[r]^-{G_0} & \Riem^+_{rot} (M) \ar [d] \\
 D^n \ar[r]^-{G_0} & \Riem^+ (M)
 }
\end{equation}
be a commutative diagram, and we have to produce a homotopy $G: [0,1] \times D^n \to \Riem^+ (M)$ such that $G(0,x)=G_0(x)$ for all $x \in D^n$ and $G(t,x) \in \Riem^+_{rot}(M)$ if $(s,x) \in (\{1\} \times D^n) \cup ([0,1] \times S^{n-1})$. Recall that $\Riem^+_{rot}(M)$ denotes the space of psc metrics which are of the form $g_N + g_0$ on $N \times B_R^k$, for some normalized rotationally invariant metric $g_0$ on $B_R^k$ with $\scal(g_0)>0$. 

The first step is an application of Proposition \ref{prop:adapting-tubularneighborhoods}.

\begin{lem}\label{lem:second-normalizationlemma}
There is a family $G'(s,x)$, $(s,x) \in [0,1] \times D^n$ of Riemannian metrics on $M$ and $r_0\in (0,R)$ such that
\begin{enumerate}
 \item the metric $G'(0,x)$ has positive scalar curvature,
 \item $G'(0,x)\in \Riem^+_{rot}(M)$ for all $x \in S^{n-1}$, 
 \item the map $G'(0,\_):(D^ n,S^{n-1})\to (\Riem^+ (M), \Riem^+_{rot}(M))$ is homotopic to $G_0$ (as a map of space pairs),
 \item for all $(s,x) \in [0,1] \times D^n$, the metric $G'(s,x)$ is normalized on the $r_0$-tube around $N$,
 \item for all $(s,x) \in ([0,1] \times S^{n-1}) \cup (\{1\} \times D^n)$, the metric $G'(s,x)$ is rotationally symmetric on the $r_0$-tube around $N$, i.e. $G'(s,x) = g_N + g(s,x)$ for some rotationally symmetric $g(s,x)$.
\end{enumerate}
\end{lem}

In short, we make the metrics $G(0,x)$ normalized on some tube, but in addition, we also take a crude interpolation of $G(0,x)$ to some rotationally invariant metric, without taking the psc condition into account.

\begin{proof}
Choose a Riemannian metric $g$ on $M$ such that $g|_{N \times B_R^k} = g_N + g'$, where $g'$ is a a rotationally symmetric normalized metric on $B_R^k$. For example, we can take $g'$ to be the euclidean metric. Let $\tilde{G}(s,x):= (1-s) G_0(x) + sg$, for $(s,x) \in [0,1] \times D^n$. We apply Proposition \ref{prop:adapting-tubularneighborhoods} to the map $\tilde{G}$ with $K= [0,1] \times D^n$ and $L= 0 \times S^{n-1} \cup 1 \times D^n$ and let $F$ be the isotopy from that Proposition. Put $G'(s,x):= F(1,s,x)^* \tilde{G}(s,x)$. This has all the desired properties.
\end{proof}

Now we replace the map $G_0$ in \eqref{proof-firstpart-diag1} by the map $G'(0,\_)$. 

For $x \in S^{n-1}$, we write $G'(0,x)= g_N + g_0(x)$ on $N \times B_R^k$. Let
\[
A:= \inf \scal (g_N) \in \bR.
\]
%is positive, then $\inf \scal_{N \times B_R^k} (G(0,x)) >A$ for all $x \in S^{n-1}$: otherwise, $g_0(x)$ won't be a psc metric. 
Choose $\eta>0$ so that 
\[
\forall x\in D^n: \inf \scal (G'(0,x))- 2\eta \geq 0
\]
and 
\[
\forall x \in S^{n-1}: \inf \scal (G'(0,x)|_{N \times B_R^k}) -2\eta \geq A. 
\]
The second condition is implied by the first one if $A \leq 0$. If $A>0$, then for each point $g \in \Riem^+_{rot} (M)$ which is of the form $g_N + g_0$ on $N \times B_R^k$, we have $\scal (g|_{N \times B_R^k}) > A$, since otherwise $g_0$ won't be a psc metric. 

Therefore, if we can produce a homotopy $G:[0,1] \times D^n \to \Riem^+ (M)$, so that 
\begin{enumerate}
\item $G(0,\_)= G'(0,\_)$, 
\item $\inf \scal (G(\lambda,x)) \geq \inf \scal (G'(0,x)) - \eta$ and
\item $G(\lambda,x)|_{N \times B_R^k}$ is of the form $g_N+ g_0(\lambda,x)$ for $(\lambda,x)\in ( [0,1] \times S^{n-1}) \cup (\{1\} \times D^n)$, with some rotationally invariant normalized metric $g_0 (\lambda,x)$,
\end{enumerate}
then $g_0 (\lambda,x)$ will have positive scalar curvature for all $(\lambda,x)\in [0,1] \times S^{n-1}$, and $G$ is a relative homotopy, and so we have finished the proof of Proposition \ref{prop:from-all-to-rotationallysymmetric}.

Next, we determine the parameters $\epsilon_0$ and $\ell$ for the Gromov--Lawson curve. 

\begin{itemize}
\item Let $\epsilon_1>0$ be small enough, so that for all $\epsilon \in (0,\epsilon_1)$ and for all $(s,x)$, the restriction of $G'(s,x)$ to the sphere $N \times S^{k-1}_\epsilon$ has scalar curvature $> \max (0,A)$. This is possible by Corollary \ref{cor:curvature-on-little-spehres}.
\end{itemize}
%The parameter $\epsilon_0$ has to be so small that for each $\epsilon\in (0,\epsilon_0)$, and each $(s,x)$, the restriction of $G(s,x)$ to the sphere $N \times S^{k-1}_\epsilon$ has scalar curvature $> \max (0,A)$. This is possible by Corollary \ref{cor:curvature-on-little-spehres}, but not yet enough: we need to know that the restriction of the metric $G'(s,x)+dt^2$ to the final bend of the GL-curve also satisfies this curvature bound. This is the purpose of the next Lemma.

%Let $\gamma_1:[0,\infty)\to \bR^2$ be a curve in the plane with the following properties: $\gamma_1 (0)= (1,0)$, during the interval $[0,\frac{\pi}{2}-\delta]$, $\gamma_1$ is the unit circle, and after $[\frac{\pi}{2}+\delta]$, $\gamma_1$ is a horizontal line. We require that in between, $\gamma_1$ is smoothened as in the proof of Proposition \ref{prop:glcurve}. In particular, the curvature is nonnegative, 

\begin{lem}\label{lem:find-small-espilon}
For all $B \in \bR$, there exists $\epsilon_0\in (0,\epsilon_1]$ such that for all $(s,x) \in [0,1] \times D^n$, the metric $(G'(s,x))_{\gamma}$ has scalar curvature at least $B$, where $\gamma$ is a curve in the plane with the following properties. 
\begin{enumerate}
\item There is $\epsilon \in (0,\epsilon_0]$ such that $0 \leq r \leq \epsilon$, $-\frac1\epsilon\leq\kappa \leq 0$ and $\theta \in [0,\frac{\pi}{2}]$, \item if $r \leq \frac{\epsilon}{\sqrt{2}}$, then $\gamma$ is a circle of radius $\epsilon$, and if $r \geq \frac{\epsilon}{\sqrt{2}}$, then $\theta \geq \frac{\pi}{4}$. 
\end{enumerate}
\end{lem}

\begin{proof}
This is an application of Lemma \ref{lem:formula}, but much easier than Proposition \ref{prop:glcurve}. Pick $r_1$ so that the curvature estimate of Lemma \ref{lem:formula} holds for all $G(s,x)$, with some constant $C$. At the points where $\gamma$ is a circle, we have $\kappa=-\frac{1}{\epsilon}$ and $\sin (\theta)\epsilon =r$. There, Lemma \ref{lem:formula} yields
\[
\scal (G'(s,x)_\gamma) \geq \scal (G'(s,x)) + \frac{1}{\epsilon^2} k (k-1) - \frac{2C}{\epsilon}.
\]
In the region $\frac{\epsilon}{\sqrt{2}} \leq r \leq \epsilon$, Lemma \ref{lem:formula} yields
\[
\scal (G'(s,x)_\gamma) \geq \scal (G'(s,x)) - \frac{2 C}{\epsilon} + \frac{(k-1)(k-2)}{2 \epsilon^2}.\qedhere
\]
\end{proof}

\begin{itemize}
\item Now we choose $\epsilon_0>0$ so that $\epsilon_0 < r_0$ and that the conclusion of Lemma \ref{lem:find-small-espilon} holds with $B=A + 2\eta$. According to Proposition \ref{prop:glcurve}, there exists a Gromov--Lawson curve $\Gamma$ with parameters $\eta$ and $\epsilon_0$, which has an inner width $r_\infty \leq \epsilon_0$. 
\end{itemize}

Let $g_{s,x}:= G'(s,x)|_{N \times S^{k-1}_{r_\infty}}$. This metric on $N \times S^{k-1}_{r_\infty}$ has scalar curvature at least $A+2\eta$ by Lemma \ref{lem:find-small-espilon}. We get a continuous map $[0,1]\times D^n \to \Riem^+ (N \times S^{k-1}_{r_\infty})$, $(s,x) \mapsto g_{s,x}$.

Next let $f: \bR \to [0,1]$ be a smooth function such that $f\equiv 0$ near $(-\infty,0]$ and $f\equiv 1$ near $[1,\infty)$ and define $b:= \max\{ \norm{f'}_{C^0},\norm{f''}_{C^0}\}$. 

For each $L>0$, we get an induced map 
\begin{equation}\label{endgame-homotopy}
D^n \times [0,1] \to \Riem (N \times S^{k-1}_{r_\infty} \times [0,L]); \; (x,\lambda)\mapsto g_{\lambda f(\frac{t}{L}),x} + dt^2. 
\end{equation}
The first two derivatives of $t \mapsto \lambda f(\frac{t}{L})$ are
\[
|\lambda f(\frac{t}{L})'| \leq  \frac{\lambda}{L}b; \; |\lambda f(\frac{t}{L})''| \leq  \frac{\lambda}{L^2}b.
\]
\begin{itemize}
\item We pick $L$ so large that $\frac{\lambda}{L}b, \frac{\lambda}{L^2}b \leq \Lambda$, where $\Lambda>0$ is the constant provided by Lemma \ref{gajer-lemma}. With these choices, we obtain
\begin{equation}\label{endgame-homotopy2}
\scal (g_{\lambda f(\frac{t}{L}),x} + dt^2) \geq A+\eta.
\end{equation}
\item Finally, put $\ell:=L+R$.
\end{itemize}

\begin{proof}[End of the proof of Proposition \ref{prop:from-all-to-rotationallysymmetric}]
We consider a diagram as in \eqref{proof-firstpart-diag1} and replace $G_0$ by $G'(0,\_)$, where $G'(s,x)$ is a family of Riemannian metrics with the properties stated in Lemma \ref{lem:second-normalizationlemma}. Let $\eps_0$ be as in Lemma \ref{lem:find-small-espilon}. 
By Proposition \ref{prop:glcurve}, there exists a Gromov--Lawson curve $\Gamma$ with parameters $\eps_0$ and $\ell$. Let $E_\Gamma: [0,1]\times M \to (M \times \{0\} )\cup (N \times \bR^k \times \bR)$ be the isotopy of embeddings determined by $\Gamma$ (as in \eqref{eq:gromovlawsonembedding}). Now we define $G(\lambda,x) \in \Riem^+ (M)$ for $\lambda \in [0,\frac12]$ by
\[
 G(\lambda,x):= E_{\Gamma ,2 \lambda}^* (G'(0,x)+dy^2)
\]
and for $\lambda \in [\frac12,1]$ by
\[
 G(\lambda,x):= E_{\Gamma ,1}^* (G'((2\lambda-1) f(\frac{y-y_5}{\ell}),x)+dy^2).
\]
By construction, $\scal(G(\lambda,x)) >0$ for all $x,\lambda$, and if $x \in S^{n-1}$, then $\scal(G(\lambda,x)) >A$. These metrics are not normalized, but the curve $t \mapsto (p,tv)$ is a variable speed geodesic. This can be rectified by a reparametrization (pull back by an isotopy of $N \times \bR^k$ which is the identity outside a compact set and which is of the form $(p,v) \mapsto (p, h_\lambda (\norm{v}) v)$ for a smooth odd function $h_\lambda$). After such a reparametrization, the metrics $G(\lambda,x)$ are normalized on the $r_0$-tube. If $x \in S^{n-1}$, they stay rotationally symmetric, and the geometric size of the region where they are does not decrease with $\lambda$. Hence after reparametrization, $G(\lambda,x)$ is rotationally symmetric and normalized on the $R$-tube, for all $x \in S^{n-1}$. This completes the proof. 

\end{proof}

\section{Rotationally symmetric metrics}\label{sec:rot}

In this section, we complete the proof of Theorem \ref{thm:main-part1-precise}. Let us first recall some notation. Let $A:= \inf\scal (g_N)\in \bR$. We choose $\delta>0$ so that $\frac{1}{\delta^2} (k-1)(k-2)+A>0$ and pick a torpedo metric $g_\torp^k$ on $\bR^k$ of radius $\delta$, which is cylindrical outside the disc of radius $R$, for some $R>0$. 
Recall that $ \Riem^+_{rot}(M) \subset \Riem^+ (M)$ is the space of all psc metrics $g$ on $M$ such that
\[
g|_{N \times B_R^k} = g_N + g_0
\]
for some rotationally symmetric normalized psc metric $g_0$ on $B_R^k$. Furthermore, $\Riem^+ (M,\varphi) \subset \Riem^+_{rot}(M)$ is the subspace of those $g$ such that $g_0=g_\torp^k$. The goal is to prove the following result, which together with Proposition \ref{prop:from-all-to-rotationallysymmetric} completes the proof of Theorem \ref{thm:main-part1-precise}. 
\begin{prop}\label{prop:main2}
The inclusion map
\[
\Riem^+(M,\varphi) \to  \Riem^+_{rot}(M) 
\]
is a weak homotopy equivalence. 
\end{prop}

\subsection{Preliminary remarks}

A rotationally symmetric normalized metric on $B_R^k$ is of the form $g=dt^2 + f(t)^2 d\xi^2$, for some \emph{warping function} $f:[0,R]\to \bR$ with the properties stated in Lemma \ref{lem:preparation}. We also recall the curvature formula 
\begin{equation}\label{eq:curvatureformula:sec3}
\sigma(f):= \scal (dt^2 + f(t)^2 d\xi^2)= (k-1)\left((k-2)\frac{1-f'^2}{f^2} - 2\frac{f''}f\right).
\end{equation}
The torpedo metric $g_\torp^k$ is given by the warping function $h_\delta$ as in \eqref{defn:warping-for-deltatorpedo}. In order for the metric $g_N + dt^2 + f(t)^2d\xi^2$ to have positive scalar curvature, we need to have $A+ (k-1)\left((k-2)\frac{1-f'^2}{f^2} - 2\frac{f''}f\right)>0$. To allow for more convenient notation when $A\leq 0$, we introduce 
\[
B:= \max\{0, -A\} \geq 0.
\]
Then 
\[
0\leq B < \frac{1}{\delta^2}(k-1)(k-2),
\]
and the condition on $f$ becomes
\[
(k-1)\left((k-2)\frac{1-f'^2}{f^2} - 2\frac{f''}f\right)>B.
\]
%These choices have the effect that $g_N + g_\torp^{\delta}$ has positive scalar curvature. Now we choose $R>0$ so that $f_\delta = \delta$ near $[R,\infty)$ (in other words, the $\delta$-torpedo is cylindrical outside the disc of radius $R$). We will show:
%Inside $N \times B_R^k$, a rotationally symmetric normalized metric is of the form $g=dt^2 + f(t)^2 d\xi^2$, for some function $f$ with the properties stated in Lemma \ref{lem:preparation}. We also recall the curvature formula 
%\begin{equation}\label{eq:curvatureformula:sec3}
%\scal(g) = (k-1)\left((k-2)\frac{1-f'^2}{f^2} - 2\frac{f''}f\right).
%\end{equation}
%In the proof of Proposition \ref{prop:main2}, we will deform a rotationally symmetric metric by deforming the warping function $f$. The psc condition which needs to be preserved during the deformation becomes 
%\[
%(k-1)\left((k-2)\frac{1-f'^2}{f^2} - 2\frac{f''}f\right) >B.
%\]
The most delicate step in the proof of Proposition \ref{prop:main2} is the following. 

\begin{prop}\label{prop:flatten}
Let 
\[
\xymatrix{
S^{n-1} \ar[d] \ar[r]^-{G_0} & \Riem^+(M,\varphi)\ar[d] \\
D^n \ar[r]^-{G_0} & \Riem^+_{rot} (M)
}
\]
be a commutative diagram. Then there exists a homotopy $G: [0,1] \times D^n \to \Riem^+_{rot} (M)$ of maps of space pairs $(D^n,S^{n-1}) \to (\Riem^+_{rot} (M),\Riem^+(M,\varphi))$ such that $G(0,\_)=G_0 (\_)$ and such that the warping function $f_{t,x}$ of $G(t,x)$ satisfies
\begin{enumerate}
\item $0 \leq f_{1,x} \leq \delta$ and $f''_{1,x} \leq 0$ on $[0,R]$, 
\item $f_{1,x}' \equiv 0$ near $R$.% $[\tau(x),R]$.
\end{enumerate}
\end{prop}

We will prove Proposition \ref{prop:flatten} in \S \ref{subsec:introduction-of-collars}, and in \S \ref{subsecd:adjustingtorpedo}, we complete the proof of Proposition \ref{prop:main2}.

%\begin{prop}\label{prop:flatten}
%Let 
%\[
%\xymatrix{
%S^{n-1} \ar[d] \ar[r]^-{G_0} & \Riem^+(M,\varphi)\ar[d] \\
%D^n \ar[r]^-{G_0} & \Riem^+_{rot} (M)
%}
%\]
%be a commutative diagram. Then there exists a homotopy $G: [0,1] \times D^n \to \Riem^+_{rot} (M)$ of maps of space pairs $(D^n,S^{n-1}) \to (\Riem^+_{rot} (M),\Riem^+(M,\varphi))$ and a function $\tau:D^n \to (0,R]$ such that $G(0,\_)=G_0 (\_)$ and such that the warping function $f_{t,x}$ of $G(t,x)$ satisfies
%\begin{enumerate}
%\item $0 \leq f_{1,x} \leq \delta$ and $f''_{1,x} \leq 0$ on $[0,R]$, 
%\item $f_{1,x}' \equiv 0$ on $[\tau(x),R]$.
%\end{enumerate}
%\end{prop}

\subsection{Introducing collars}\label{subsec:introduction-of-collars}

In order to prove Proposition \ref{prop:flatten}, we will change the warping function $f$ by composition with another function $h$ or a $1$-parameter family thereof. The composition $f \circ h$ will have a different domain of definition. In order to obtain a well-defined family of Riemannian metrics on $M$, we introduce the following construction.

We fix, once and for all, diffeomorphisms $\varphi_{a,b}$ of $(0,\infty)$ for each $0<a\le b$ such that
\begin{itemize}
\item $\varphi_{a,b}(b) = a$,
\item $\varphi_{a,a} = \id$,
\item $\varphi_{a,b}|_{[0,\frac a2]\cup[2b,\infty)}=\id$,
\item $\varphi_{a,b}'\equiv 1$ near $b$,
\item $\varphi_{a,b}$ depends smoothly on $a,b$.
\end{itemize}
The formula $\phi_{a,b}(x,v) := (x,\frac{\varphi(\norm{v})}{\norm{v}}v)$ defines diffeomorphisms of $N\times\bR^k$. These are compactly supported and can be extended by the identity to $M$. 

\begin{lem}\label{lem:deformationnull}
Let $g \in \Riem_{rot}^+ (M)$ given by $g=g_N+dt^2 + f(t)^2d\xi^2$ on $N\times B_R^k$. Let $h\colon [0,\infty) \to [0,\infty)$ be smooth with $h(0)=0$, $h'\equiv1$ near $0$ and $0\le h'\leq 1$. Let $S\in(0,\infty)$ be such that $h(S)=R$ (this enforces $S \geq R$). Then the formula
$$
\Lambda(g,h,S) := \begin{cases}
	g_N + dt^2 + f(h(t))^2d\xi^2 &\text{ on } N\times B_S^k,\\
	{\phi_{R,S}}^*g	&\text{ else, }
\end{cases}
$$ 
defines a smooth Riemannian metric on $M$ in each of the following cases:
\begin{enumerate}
 \item $f'\equiv 0$ near $R$, or
 \item $h'\equiv 1$ near $S$.
\end{enumerate}
\end{lem}

\begin{proof}
We need to show that $g_N+dt^2 + f(h(t))^2d\xi^2$ and ${\phi_{R,S}}^*g$ coincide on $N\times B_S\setminus N\times B_{S-\epsilon}$ for some $\epsilon>0$. But near $N \times S^{k-1}_S$, we have 
\[
 {\phi_{R,S}}^*g = g_N + \varphi_{R,S}'(t)^2dt^2 + f(\varphi_{R,S}(t))^2d\xi^2\overset{\text{near }S}= g_N + dt^2 + f(t-S+R)d\xi^2
 \]
because ${\varphi_{R,S}}'\equiv 1$ and $\varphi_{R,S}(t)=t-S+R$ near $S$. Now in either of the two cases we have $f(t-S+R) = f(h(t))$ for $t$ near $S$: If $f'=0$ near $R$, we have $f(t-S+R) = f(R)$ and $h(t)$ is close to $R$ near $S$; and if $h'=1$ near $S$, then $h(t)=t-S+R$ near $S$. 
\end{proof}

Let us record some further simple properties of this construction. We omit the easy proof.

\begin{lem}\label{lem:deformation}\mbox{}
\begin{enumerate}
\item In the situation of Lemma \ref{lem:deformationnull}, $\Lambda(g,h,S)$ is rotationally symmetric and normalized on $B_S^k \supset B_R^k$. 
\item Let $X$ be a space and let $g\colon X\to \Riem_{rot}(M)$, $h\colon X\to C^\infty([0,\infty),\bR)$ and $S\colon X\to (0,\infty)$ be continuous maps such that $h(x)$ and $S(x)$ satisfy the requirements of Lemma \ref{lem:deformationnull} and assume that for each $x \in X$, one of the two conditions from Lemma \ref{lem:deformationnull} is satisfied. Then $X\to \Riem_{rot}(M), x\mapsto\Lambda(g(x),h(x),S(x))$ is continuous.
\end{enumerate}
\end{lem}

From now on, we only change the warping function inside the $R$-disc. Note that the metric $\Lambda (g,h,S)$, restricted to the complement of $N \times B_S^k$, is isometric to the metric $g$. Hence we only need to control the scalar curvature of $\Lambda(g,h,S)$ inside $B_S^k$, where it is determined by \eqref{eq:curvatureformula:sec3}. In particular, our consideration will only involve the metrics on $B_S^k$, not on $N$. 
Let us make a few observations: If $\scal (dt^2 + f(t)^2 d\xi^2)\geq B'>0$, then L'H\^{o}pital's rule shows that
\[
 B' \leq  \lim_{t\to 0}(k-1)\left((k-2)\frac{1-f'(t)^2}{f(t)^2}-2\frac{f''(t)}{f(t)}\right) = (k-1)\lim_{t\to 0}\left((k-2)\frac{-2f'(t)f''(t)}{2f(t)f'(t)} - 2\frac{f''(t)}{f(t)}\right) =
\]
\[
 = -(k-1)k\lim_{t\to 0}\frac{f''(t)}{f(t)} = -k(k-1)\lim_{t\to 0}\frac{f'''(t)}{f'(t)} = -k(k-1)f'''(0)
\]
and hence
\begin{equation}\label{equ:thirdderivative}
f'''(0)\le\frac{-B'}{k(k-1)} .
\end{equation}
If $h$ is a function as in Lemma \ref{lem:deformationnull}, then the scalar curvature of $dt^2 + (f\circ h)^2d\xi^2$ is given by
\[
(k-1)\left((k-2)\frac{1-f'(h)^2h'^2}{f(h)^2} - 2\frac{f''(h)}{f(h)}h'^2 - 2\frac{f'(h)}{f(h)}h''\right)= 
\]
\begin{equation}\label{eq:curvature-convexcombination}
h'^2(k-1)\left((k-2)\frac{1-f'(h)^2}{f(h)^2} - 2 \frac{f''(h)}{f(h)}\right) + (1-h'^2)\frac{(k-1)(k-2)}{f(h)^2} - 2(k-1)\frac{f'(h)}{f(h)}h'', 
\end{equation}
using the self-explanatory notation $f(h):= f \circ h$.
\begin{lem}\label{lem:easyestimate}
Let $f$ be the warping function of a metric $g$ satisfying $f'\in[0,1]$ and $\scal(g)\ge B''>B'>0$ and let $h$ be as above. Assume also that for some $r>0$
\begin{enumerate}
\item $B''f^2\le (k-1)(k-2)$ on $[0,r]$ and 
\item $h''\le\frac12\frac{B''-B'}{k-1}f$ whenever $h\le r$, say $h([0,s])\subset[0,r]$.
\end{enumerate}
Then $\scal(dt^2 + f(h(t))^2d\xi^2)\ge B'$ on $[0,s]$.
\end{lem}

\begin{proof}
From \eqref{eq:curvature-convexcombination}, we get $\scal(dt^2 + f(h(t))^2d\xi^2)\ge$
\[
 h'^2B'' + (1-h'^2)B'' -f'(h)(B''-B') \ge B''-(B''-B') = B'. \qedhere
\]
\end{proof}
\begin{rem}\label{rem:easyestimate}
If $f=h_\delta$ is the torpedo function of radius $\delta$ and if $B''\le\frac1{\delta^2}(k-1)(k-2)$, then hypothesis (1) of \ref{lem:easyestimate} is satisfied for each $r>0$. 
\end{rem}
%\subsection{Flattening of the warping function}
The following two elementary lemmas are slightly adapted versions of \cite[Lemma 3.5 and Lemma 3.7]{Chernysh}.

\begin{lem}[Existence of sloping functions]\label{lem:slopingcurve}
Let $b \in (0,R)$, $0<a<\frac8{10}b$ and let $p>0$. Then there exists $q \in (0,1)$, only depending on $p$ and $b$, a family $u_{r,s}:\bR \to \bR$ of functions and $c_{r,s}\in\bR$, both depending continuously on $(r,s)\in [0,1] \times [0,1]$ such that
\begin{enumerate}
\item $u_{r,0}=id$ and $u_{r,s}=id$ on $(-\infty,\frac8{10}a]$ for all $r,s$,
%\item $u_{r,s}(\frac8{10}c_{r,s}) = \frac8{10}b$,
\item $u_{r,s}(c_{r,s}) = b$,
\item $u_{r,s}''\le pr$,
\item $u_{r,s}''|_{[\frac8{10}a,a]}\leq 0$, $u_{r,s}''|_{[\frac8{10}c_{r,s},c_{r,s}]} \geq 0$ and $u_{r,s}''=0$ outside these intervals,
\item $u_{r,s}'=1-sq$ on $[a,\frac8{10}c_{r,s}]$ and $u_{r,s}' = (1-sq+rsq)$ on $[c_{r,s},\infty)$,
\item $0\leq u_{r,s}' \leq 1$ for all $r,s$. 
\end{enumerate}
%Furthermore, there exists a universal constant $\rho>0$ such that we can choose $q:=\rho a p$. 
We call $u_{r,s}$ a \emph{sloping function} with parameters $a,b,p$ and $q$ the \emph{resulting slope}. The situation is depicted in the following figure.
\begin{figure}[h]
\includegraphics[width=35em]{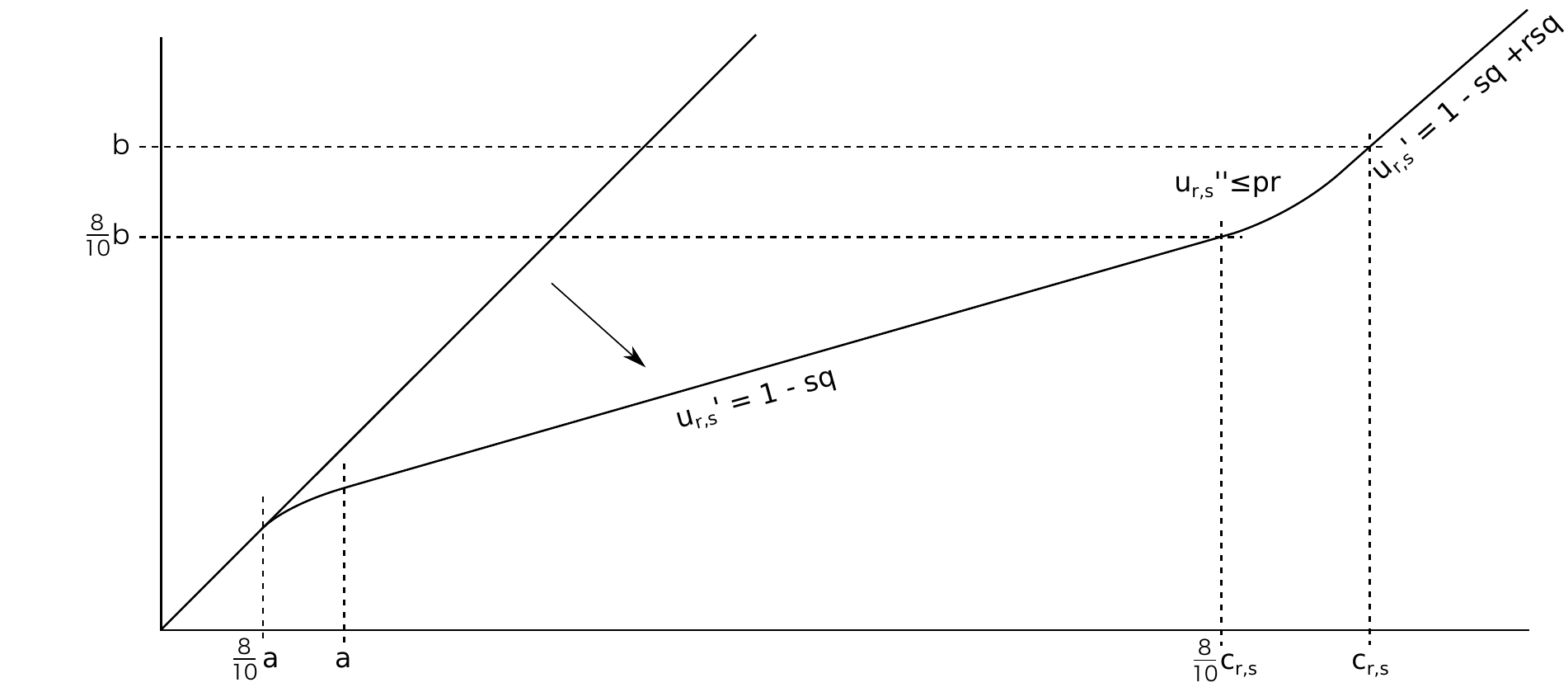}
\end{figure}
\end{lem}

\begin{proof}
We construct $u_{r,s}$ by constructing its second derivative, and do this by first constructing a piecewise continuous approximation to the second derivative of $u_{r,s}$. Choose $q\leq  \frac{b}{10}p$. Define a piecewise continuous function $w_{r,s}: \bR\to \bR$ by 
\[
 w_{r,s} = -s \chi_{[\frac{17}{20}a ,\frac{19}{20}a]} \frac{10 q}{a} + rs \chi_{[\frac{17}{20}e_{r,s} ,\frac{19}{20}e_{r,s}]} \frac{10 q}{e_{r,s}}.
\]
where $e_{r,s}$ is to be determined. Let $v_{r,s} (x):= \int_0^x \int_0^t w_{r,s}(y)dy$ and choose $e_{r,s} \geq b$ to be the unique point such that $v_{r,s}(\frac{8}{10} e_{r,s})= \frac{8}{10} b$. Now let $\xi \geq 0$ be a smooth, even, nonnegative function with compact support and $\int_\bR\xi (x) dx =1$ and let $\xi_\eps (x):= \frac{1}{\eps}\xi(\frac{x}{\eps})$. The function 
\[
 u_{r,s} (x):=\xi_\eps \ast v_{r,s}% \int_0^x (\int_0^y \xi \ast w_{r,s}(t) dt )dy
\]
has, if $\eps$ is sufficiently small, all desired properties. Finally, we define $c_{r,s}$ as the unique point with $u_{r,s}(c_{r,s})=b$.
\end{proof}

\begin{lem}[Existence of bending functions]\label{lem:bendingcurve}
Let $C>0$ and $\beta>0$. Then there exists $\alpha\in(0,\beta)$ and a family of functions $v_{r,s}:\bR \to \bR$ and $d_{r,s}\in\bR$, depending continuously on $(r,s) \in (0,1] \times [0,1]$ such that
\begin{enumerate}
\item $v_{r,0}=id$ and $v_{r,s}=id$ on $(-\infty,\frac{\alpha}2]$,
\item $v_{r,s}(d_{r,s})=\beta$,
\item $v_{r,1}'\equiv 0$ near $\alpha$,
\item $v_{r,s}''\le C\frac{r}{t}$ and $v_{r,s}''\le 0$ on the complement of $[2\alpha,d_{r,s}]$,
\item $v_{r,s}'\equiv 1-s+rs$ on $[d_{r,s},\infty)$.
\item $0\leq v_{r,s}' \leq 1$ for all $r,s$. 
\end{enumerate}
We call $v_{r,s}$ a family of bending functions with parameters $C,\beta$. The point $\alpha$ is called the \emph{attacking point}. The following figure depicts the situation.
\begin{figure}[h]
\includegraphics[width=35em]{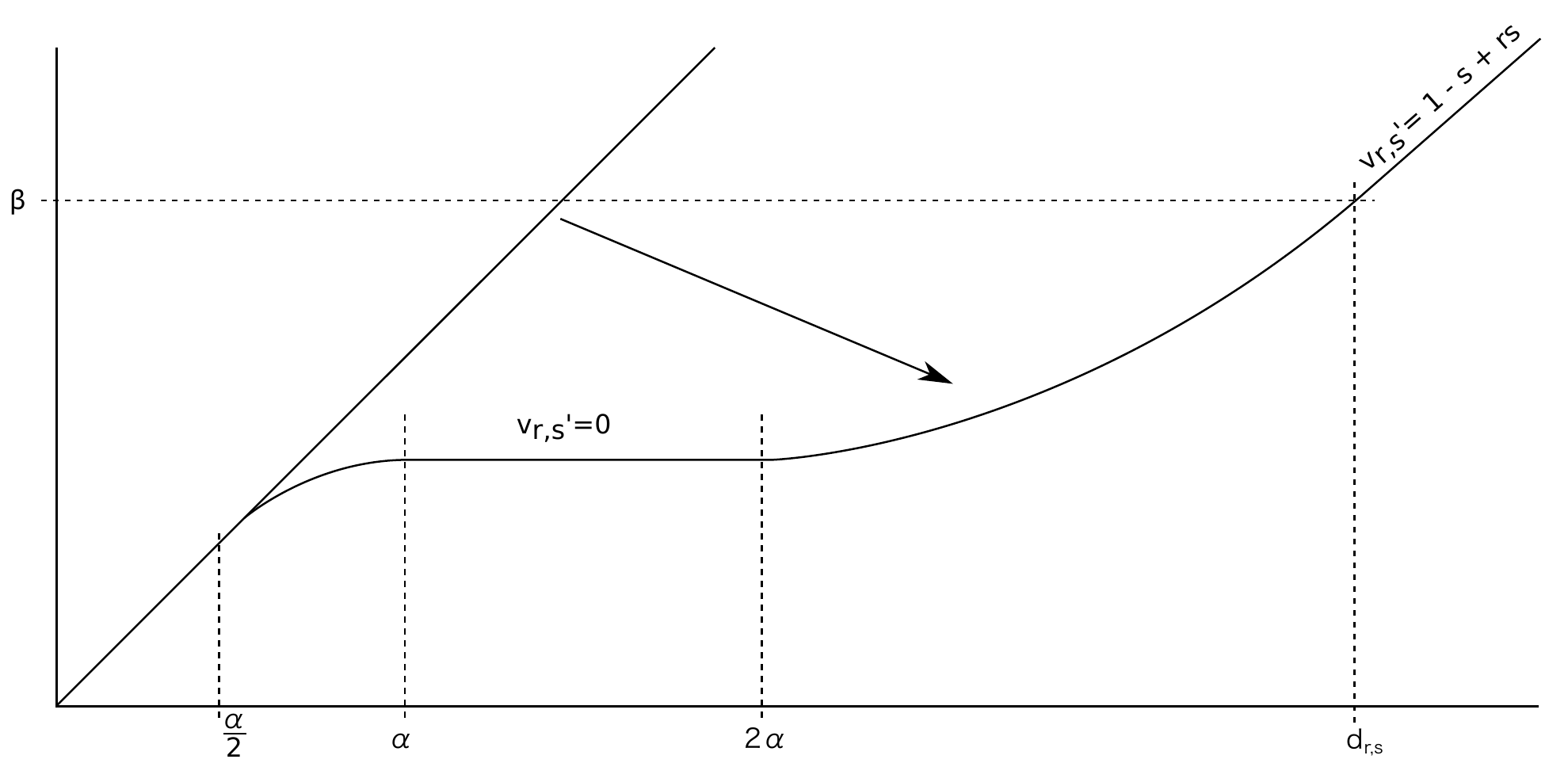}
\end{figure}
\end{lem}

\begin{proof}
This is by a similar method as the proof of Lemma \ref{lem:slopingcurve}. Let $\gamma:= \frac{\beta}{2}$ and $\alpha:= \frac{\beta}{4} e^{- \frac{1}{C}}$. We consider the piecewise continuous function
\[
 f_{r,s} = -s \chi_{[\frac{4}{6} \alpha ,\frac{5}{6}\alpha ]} \frac{6}{\alpha} + rs \chi_{[2 \alpha,\gamma]} \frac{C}{t} 
\]
and set 
\[
 w_{r,s}(t):= \int_0^t \Bigl( 1 + \int_0^x   f_{r,s}(y)dy \Bigr) dx.
\]
A straightforward, but lenghty integral computation reveals that $w_{r,s}$ has all the desired properties, except that it is only piecewise $C^2$. For suffienctly small $\eps$, consider the convoluted function
\[
v_{r,s}:= \xi_\eps \ast w_{r,s}.
\]
The point $d_{r,s}$ is the unique point with $v_{r,s}(d_{r,s})=\beta$.
\end{proof}

\begin{proof}[Proof of Proposition \ref{prop:flatten}]
We may assume that  $G_0(x) = G_0(\frac{x}{\norm{x}})$ for $\norm{x}\ge\frac12$, otherwise we perform an obvious homotopy beforehand to achieve this. 
We write the metric $G_0 (x)$ as $g_N + g(x)$ for a rotationally invariant metric $g(x)$ on $B_R^k$. Recall that $B \in [0,\frac{1}{ \delta^2} (k-1)(k-2))$ and $\scal (g(x)) > B$ for all $x \in D^n$. Let $f_x: [0,R] \to \bR$ be the warping function of $G_0 (x)$. The proof begins with making some choices:
\begin{enumerate}
\item Pick 
\[
 B < B' < B'' < \frac{1}{\delta^2} (k-1)(k-2)
\]
so that $\scal(g(x)) \geq B''$ for all $x \in D^n$. 
\item By \eqref{equ:thirdderivative}, there is $S \in (0,R)$ such that $f_x'\in[0,1]$ and $f_x''\le 0 $ on $[0,S]$ for all $x \in D^n$. We can furthermore pick $S$ so that $B'' S^2 \leq  (k-1)(k-2)$ and $S \leq \delta$.
\item Let $F:= \inf_{x\in D^n}f_x(\frac{8}{10}S)>0$ and $p:=\frac{1}{2(k-1)}(B''-B')F\leq \frac{8}{10}\frac{1}{2(k-1)}(B''-B')S$.
\item Let $q$ be the resulting slope (see Lemma \ref{lem:slopingcurve}) of a sloping function with parameters $(\alpha,S,p)$, for $\alpha <  \frac{8}{10} S$ (recall the constant $q$ only depends on $p$ and $S$, not on $\alpha$, which we have to pick later).
\item Now we pick $T \in (0, \frac{8}{10} S]$ and $C>0$ so that 
\[
 B' T^2 + 2(k-1) C \leq (k-1)(k-2) (1-(1-q)^2).
\]
\item Pick a family of bending functions $v_{r,s}$ for the parameters $(C,T)$ and let $\alpha>0$ be the attacking point of this family of bending functions. The numbers $d_{r,s} \in \bR$ are as in Lemma \ref{lem:bendingcurve}.
\item Pick a family of sloping functions $u_{r,s}$ for the parameters $(\alpha,S,p)$. The numbers $c_{r,s}$ are as in Lemma \ref{lem:slopingcurve}.
%Choose $\alpha \in (0,\frac{8}{10} S]$ (to be chosen more carefully later on) and let $u_{r,s}$ be a family of sloping functions with parameters $(\alpha,S,p)$ and resulting slope $q$.
\end{enumerate}
Now we construct the homotopy $G: [0,1] \times D^n \to \Riem^+_{rot}(M)$. We first construct it on the part $[0,1] \times D^n_{\frac{1}{2}}$, the disc of radius $\frac {1}{2}$. 
\begin{enumerate}
 \item On $[0, \frac{1}{3}] \times D^n_{\frac{1}{2}}$, we define
 \[
  G(\lambda,x):= \Lambda(g(x), u_{1,3\lambda},c_{1,3\lambda})
 \]
and claim that $\scal(G(\lambda,x)) \geq B'$ for all such $x$ and $\lambda$. In the region where $u_{1,3 \lambda}'' \leq 0$, there is no problem: there $f_x(u_{1,3\lambda}) \leq S$, and we picked $S$ small enough to satisfy the first hypothesis of Lemma \ref{lem:easyestimate}. In the region where $u_{1,3 \lambda}'' \geq 0$, we have by construction $u_{1,3 \lambda} \geq \frac{8}{10} S$, and hence $f_x(u_{1,3\lambda}) \geq F$. Therefore 
\[
 u_{1,3\lambda}'' \leq p = \frac{F(B''-B')}{2(k-1)} \leq \frac{B''-B'}{2(k-1)} f_x
\]
and the claim follows from Lemma \ref{lem:easyestimate}.
 \item The warping function $f=f_{x,\frac{1}{3}}$ of $G(x,\frac{1}{3})$ satisfies $f' \leq 1-q$ on the interval $[\alpha,T]$. Furthermore, the scalar curvature of $G(x,\frac{1}{3})$ is bounded from below by $B'$. Now we define $G: [\frac{1}{3},\frac{2}{3}] \times D^n_{\frac{1}{2}}$ by
 \[
 G(\lambda,x) := \Lambda(G(\frac13,x),v_{1,3\lambda-1}, d_{1,3\lambda-1}).
 \]
We claim that $\scal(g(x,\lambda)) \geq B'$ for all such $\lambda$ and $x$. Again, there is no problem in the region where $v_{1,3\lambda-1}'' \leq 0$. In the region where $v_{1,3\lambda-1}'' \geq 0$, we have $f' \leq 1-q$, $0 \leq v' \leq 1$, $f'' \leq 0$ and $f \leq T$. Therefore 
\[
 (k-1) \Bigl( (k-2) \frac{1-f'^2 v'^2}{f^2} - 2 \frac{f''}{f} v'^2 - 2 \frac{f'}{f} v'' \Bigr) - B' \geq 
\]
\[
\geq (k-1) \Bigl( (k-2) \frac{1-(1-q)^2}{f^2}  - 2 \frac{f'}{f} v'' \Bigr) - B' =
\]
\[
=\frac{1}{f^2}   \Bigl( (k-1)(k-2)(1-(1-q)^2)  - 2(k-1) f f' v''  - B' f^2 \Bigr). 
\]
But now $f' \leq 1$ in the relevant region, which implies $f(t) \leq t$. Since $0 \leq v'' \leq \frac{C}{t}$, the last term is 
\[
 \geq  \frac{1}{f^2}   \Bigl( (k-1)(k-2)(1-(1-q)^2)  - 2 (k-1)t \frac{C}{t}  - B' T^2 \Bigr) = \frac{1}{f^2}   \Bigl( (k-1)(k-2)(1-(1-q)^2)  - 2 (k-1) C  - B' T^2 \Bigr),
\]
using that $f \leq T$. But this is nonnegative, by our choice of $T$. 
\item Now we turn to the region $\norm{x} \geq \frac{1}{2}$. Here we have $f_x=h_{\delta}$, the torpedo function of radius $\delta$. In the region $\frac12 \leq \norm{x} \leq \frac23$ and $0 \leq \lambda \frac23$, we merely change the point where the warping function obtained by composition with $v_{...}$ or $u_{...}$ is glued to the original metric, until this gluing is done in the region where the warping function $f_x=h_\delta$ is constant. There is no problem in doing this, as all the functions $u_{1,s}$ and $v_{1,s}$ are linear with slope $1$ beyond $c_{1,s}$ and $d_{1,s}$. The concrete realization by formulas is 
\[
G(\lambda,x):=
\begin{cases}
\Lambda (g(x),u_{1,3\lambda}, c_{1,3\lambda}+ 6 (R-S)\norm{x} + 3(S-R)) & \lambda \leq \frac13\\
\Lambda (g(\frac13,x),v_{1,3\lambda-1},d_{1,3\lambda-1}- 3 (c_{1,1}+ R-S-T) + 6 (c_{1,1}+R-S-T) \norm{x}) & \frac13 \leq \lambda \leq \frac23.
\end{cases}
\]
For $\norm{x}=\frac23$, we have
\[
G(\lambda,x):=
\begin{cases}
\Lambda (g(x),u_{1,3\lambda}, c_{1,3\lambda}+ R-S) & \lambda \leq \frac13\\
\Lambda (g(\frac13,x),v_{1,3\lambda-1},d_{1,3\lambda} + c_{1,1}+R-S-T )& \frac13\leq \lambda \leq \frac23,
\end{cases}
\]
and since $u_{1,3\lambda} (c_{1,3\lambda} + R-S)=R$ and $v_{1,3\lambda-1}(d_{1,3\lambda-1} + c_{1,1}+R-S-T) =T+ c_{1,1}+R-S-T=  c_{1,1}+R-S$, the gluing now takes place in the region where $f_x = h_\delta$ is constant. Hence (compare Lemma \ref{lem:deformationnull}), we are now free to change the functions $u_{1,3\lambda}$ and $v_{1,3\lambda-1}$ by functions whose derivative at the relevant point is not equal to $1$. We use this additional freedom to construct the homotopy in the region $\norm{x} \geq \frac23$. 
\item In the region $\frac23 \leq \norm{x} \leq \frac56$, we use the first parameter in the sloping and bending function and ``dampen'' those. To that end, let us pick $\eta\in (0,1)$ (we have to pick $\eta$ small enough so that the next step goes through). We set, for $\frac23\leq \norm{x} \leq \frac56$,
\[
G(\lambda,x):=
\begin{cases}
\Lambda (g(x),u_{6(\eta-1)\norm{x}+5-4\eta,3\lambda}, \tilde{c}_{\norm{x},3\lambda}) & \lambda \leq \frac13\\
\Lambda (g(\frac13,x),v_{6(\eta-1)\norm{x}+5-4\eta,3\lambda-1},\tilde{d}_{\norm{x},3\lambda}-1 )& \frac13\leq \lambda \leq \frac23,
\end{cases}
\]
where $\tilde{c}_{\norm{x},3\lambda}$ is the unique point with $u_{6(\eta-1)\norm{x}+5-4\eta,3\lambda}( \tilde{c}_{\norm{x},3\lambda})=R$ and $\tilde{d}_{\norm{x},3\lambda-1}$ is the unique point with $v_{6(\eta-1)\norm{x}+5-4\eta,3\lambda-1}(\tilde{d}_{\norm{x},3\lambda} -1)= \tilde{c}_{\norm{x},1}$. 

All the curvature estimates done in this proof so far apply as well when the sloping function $u_{1,3\lambda}$ is replaced by $u_{r,3\lambda}$ and the bending function $v_{1,3\lambda-1 }$ is replaced by $v_{r,3\lambda-1}$, for $r >0$. Hence the above formula defines metrics of positive scalar curvature. 

For $\norm{x}=\frac56$, we can write 
\[
G(\lambda,x)=\Lambda (g(x),h_\lambda, a_\lambda),
\]
where 
\[
h_\lambda:= 
\begin{cases}
u_{\eta,3\lambda} & \lambda \leq \frac13\\
u_{\eta,1} \circ v_{\eta, 3\lambda-1} & \lambda \geq \frac13,
\end{cases}
\]
and $a_\lambda\in (0,\infty)$ is the point with $h_\lambda(a_\lambda)=R$. By Lemma \ref{lem:slopingcurve}, Lemma \ref{lem:bendingcurve} and the chain rule, we have 
\begin{equation}\label{est:secderh}
h''_\lambda \leq \max\{\eta p , \eta\frac{C}{2 \alpha} \} = \eta\max\{ p , \frac{C}{2 \alpha} \}. 
\end{equation}
\item In the region $\frac56 \leq \norm{x}\leq 1$, we change the function $h_\lambda$ and ``pull it away from zero to the region around $R$''. More precisely, we set 
\[
h_{\lambda,s}(t):= h_\lambda (t-s)+s
\]
and let $a_{\lambda,s}$ be the point with $h_{\lambda,s}(a_{\lambda,s})=R$. Now we put, for $\lambda \in [0,\frac23]$ and $\frac56 \leq \norm{x} \leq 1$:
\begin{equation}\label{last-homotopy}
G(\lambda,x)=\Lambda (g(x),h_{\lambda,6R\norm{x}-5R}, a_{\lambda,6R\norm{x}-5R}). 
\end{equation}
\begin{center}
\begin{figure}
\includegraphics[width=30em]{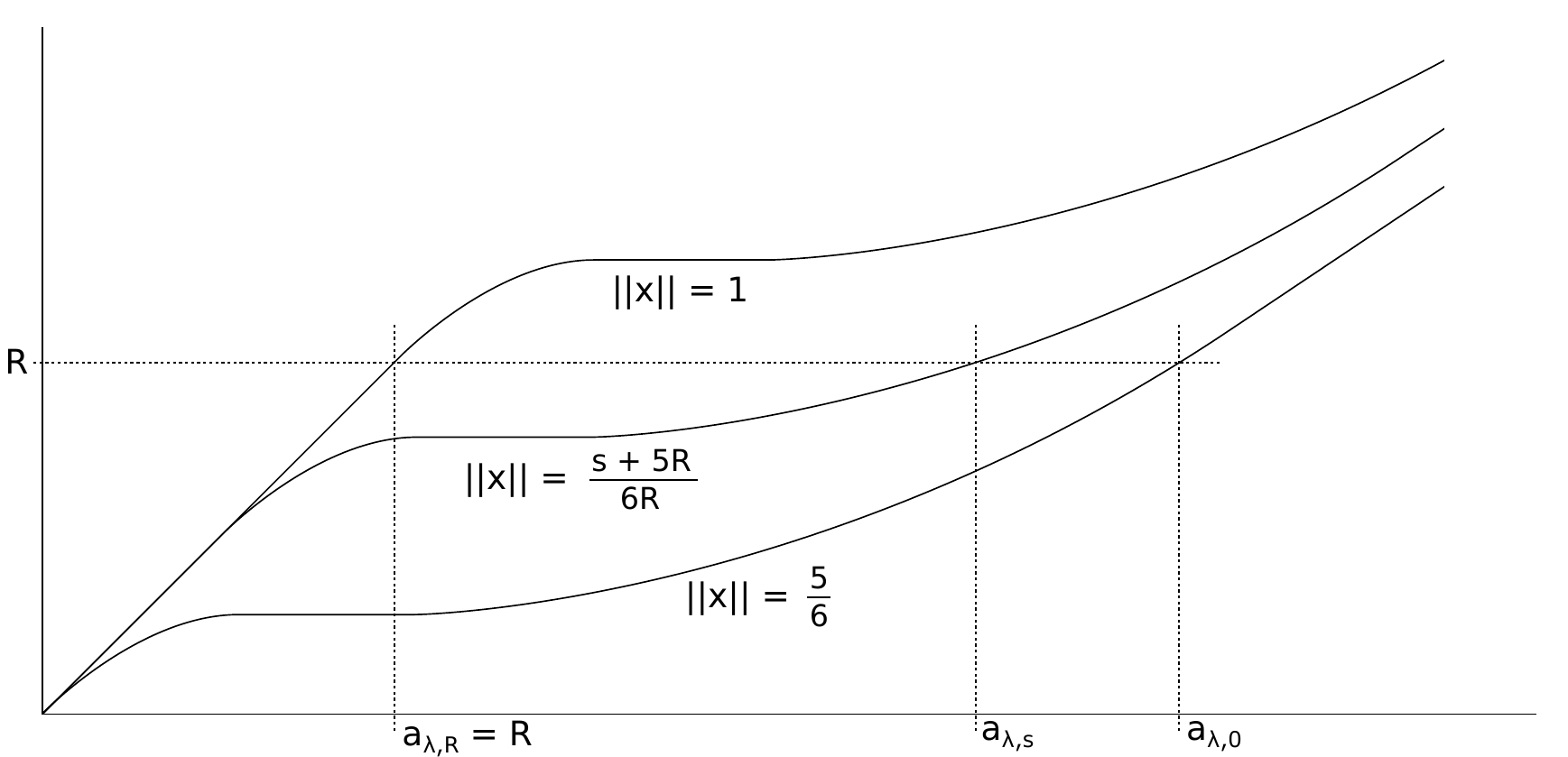}
\caption{$h_{\lambda,6R\norm{x}-5R}$ for various values of $\norm{x}$}
\end{figure}
\end{center}
If $\norm{x}=1$, then $h_{\lambda,6R\norm{x}-5R}\equiv \id$ on $[0,R]$. In other words, the metric $G(\lambda,x)$ coincides with the original metric on $B_R^k$ (but it is changed outside this disc), so that we indeed get a relative homotopy. It remains to show that \eqref{last-homotopy} defines a psc metric, and for this, we have to pick $\eta$ sufficiently small. The functions $h_{\lambda,s}$ are just translated versions of $h_\lambda$, and so their second derivatives still satisfies \eqref{est:secderh}. But recall Lemma \ref{lem:easyestimate} and Remark \ref{rem:easyestimate}: together, they show that \eqref{last-homotopy} defines a psc metric, as long as we pick $\eta$ small enough so that
\[
 \eta\max\{ p , \frac{C}{2 \alpha} \} \leq \frac{B''-B'}{2(k-1)}  \delta.
\]
\item Let us summarize what we have achieved so far: The map $G: [0,\frac{2}{3}] \times D^n\to \Riem^+_{rot}(M)$ is continuous, $G([0,\frac23]\times S^{n-1})\subset \Riem^+ (M,\varphi)$. For $\lambda = \frac23$, the metric $G(x,\frac{2}{3})$ has a warping function $f_{\frac23,x}$, which we constructed in such a way that $f_{\frac23,x}'\equiv 0$ near 
\[
\tau(x):=
\begin{cases}
\alpha & \norm{x} \leq \frac{5}{6}\\
\max{R,\alpha + 6R \norm{x} - 5R  } & \norm{x} \geq \frac{5}{6}.  
\end{cases}
\]
\item The last step is to stretch the collar around $\tau(x)$. This does not require us to be careful anymore. Let $\eps>0$ so that $f_{\frac23,x}' \equiv 0$ on $[\tau(x)-\eps,\tau(x)+\eps]$. Pick a family $(h_{\lambda,x})_{(\lambda,x)\in [\frac23,1]\times D^n}$ of smooth functions and $a_{\lambda,x}\in \bR$, depending continuously on $\lambda$ and $x$ so that 
\begin{itemize}
\item $h_{\frac23,x}=\id$, 
\item $h_{\lambda,x} |_{[0,\tau(x)-\eps]}=\id$ for all $\lambda$, 
\item $h_{\lambda,x}|_{[\tau(x)+\eps,\infty)}\equiv \tau(x)$ for $\lambda\geq \frac56$,
\item $h''_{\lambda,x}\leq 0$, $0 \leq h'_{\lambda,x} \leq 1$ for all $\lambda$.
\item $a_{\lambda,x}= \tau(x)$ for $\lambda \in [\frac23,\frac56]$,
\item $a_{1,x}=R$ for all $x$, 
\item $a_{\lambda,x}$ is monotone increasing in $\lambda$.
\end{itemize}
The desired deformation of the metric on $[\frac23,1]\times D^n$ is
\[
G(\lambda,x):= \Lambda (g(\frac23,x), h_{\lambda,x} , a_{x,\lambda}). 
\]
%\mnote{GF: Shall we include pictures for step 5 and 7?}
\end{enumerate}
\end{proof}

\subsection{Deforming to a torpedo and adjusting widths}\label{subsecd:adjustingtorpedo}

Now we are able to finish the proof of Proposition \ref{prop:main2} and hence of Theorem \ref{thm:main-part1-precise}. Consider a commutative diagram
\begin{equation}\label{diag:adjustingwidths}
 \xymatrix{
 S^{n-1} \ar[r]^-{G_0} \ar[d] & \Riem^+ (M,\varphi)\ar[d]\\ 
 D^n \ar[r]^-{G_0} & \Riem^+_{rot}(M).
 }
\end{equation}
By Proposition \ref{prop:flatten}, we may assume that $G_0(x)$ is given, on $N \times B_R^k$, by $g_N + dt^2+ f_x (t)^2 d\xi^2$, where the warping function $f_x$ satisfies
\begin{itemize}
\item $f_x' \equiv 0$ on near $R$, 
\item $f_x \leq \delta$, and we define $\delta_x:= f_x (R)$,
\item $0 \leq f'_x \leq 1$ and $f_x'' \leq 0$ on $[0,R]$,
\item for $\norm{x}\geq \frac12$, $f_x$ is the $\delta$-torpedo function $h_\delta$ (in \ref{prop:flatten}, this is only required for $\norm{x}=1$, but an obvious homotopy achieves this condition for $\norm{x} \geq \frac12$). 
\end{itemize}
Furthermore, there are constants $ \frac{(k-1)(k-2)}{\delta^2} >B''>B' >B\geq 0$ so that 
\[
 (k-1) \Bigl( (k-2) \frac{1-f_x'^2}{f_x^2} - 2 \frac{f_x''}{f_x} \Bigr) \geq B''.
\]
Let us make three observations.

\begin{observation}\label{firstobsection}
By a collar stretching homotopy as in the last step of the proof of Proposition \ref{prop:flatten}, we may also assume that the metrics $G_0 (x)$ are rotationally symmetric on a bigger disc $B_{R_\infty}^k$ and the warping function is constant on $[R,R_\infty]$, for $R_\infty$ as large as we want. For a given $R_\infty>R$, pick a smooth function $a: \bR \to [0,1]$ such that $|_{(-\infty,R]} \equiv 0$ and $a|_{[R_\infty,\infty)} \equiv 1$. For $\beta\in(0,\delta]$ and $p,q\in [\beta,\delta]$, let 
\[
a_{p,q} (t):= (1-a(t))p + a(t) q.
\]
Then
\begin{align*}
	\sigma(a_{p,q})& =  (k-1)\left((k-2)\frac1{a_{p,q}^2} - (k-2)\frac{a'{}_{p,q}^2}{a_{p,q}^2} - 2\frac{a''_{p,q}}{a_{p,q}}\right)\\
		&\ge(k-1)\frac{k-2}{\delta^2} - \frac{(k-1)(k-2)(q-p)^2a'{}^2}{\beta^2} + \frac{|p-q||a''|}{\beta}\\
		&\ge\underbrace{\frac{(k-1)(k-2)}{\delta^2}}_{\ge B''} - \delta^2\frac{(k-1)(k-2)a'{}^2}{\beta^2} + \delta\frac{|a''|}{\beta}
\end{align*}
So, if $R_\infty$ is large enough, there exists a function $a$ such that the latter term is at least $B'$.
%Let us first determine the value of $R_\infty$ we need. For a given $R_\infty>R$, pick a smooth function $a: \bR \to [0,1]$ such that $|_{(-\infty,R]} \equiv 0$ and $a|_{[R_\infty,\infty)} \equiv 1$. For $p,q\in (0,\delta]$, let 
%\[
%a_{p,q} (t):= (1-a(t))p + a(t) q.
%\]
%Then 
%\[
%\sigma (a_{p,q}) \geq \frac{(k-1)(k-2)}{\delta^2} - (k-1)(k-2) a'^2- 2 |a''| (k-1) \geq B''- (k-1)(k-2) a'^2- 2 |a''| (k-1),
%\]
%and we pick $R_\infty$ so large that there exists such a function $a$ so that the latter term is at least $B'$. 
\end{observation}

\begin{observation}\label{secondobsection}
If $f_0,f_1$ are two warping functions such that $0 \leq f'_i \leq 1$ and $f_i''\leq 0$, and such that $\sigma(f_i)>0$ on $[0,R]$, then for any $\lambda \in [0,1]$, $\sigma ((1-\lambda)f_0+\lambda f_1) >0$ on $[0,R]$. This follows from \eqref{eq:curvatureformula:sec3} and \eqref{equ:thirdderivative}. Since by \eqref{equ:thirdderivative}, we have $((1-\lambda)f_0(0)+\lambda f_1(0))'''<0$, it follows that $((1-\lambda)f_0(t)+\lambda f_1(t))'<1$ for $t>0$ and $((1-\lambda)f_0(t)+\lambda f_1(t))''\leq 0$. Hence the scalar curvature is positive. 
\end{observation}

\begin{observation}\label{thirdobsection}
For $0<\theta$ and a warping function $f$, let $f^\theta(t) := \theta f(\frac t\theta)$. Then $f^\theta(t)' = f'(\frac t\theta)$ and $f^\theta(t)'' = \frac1\theta f''(\frac t\theta)$ and we get 
$$\sigma(f^\theta(t)) = \frac{1}{\theta^2}(k-1)\left((k-2)\frac{1-f'(\frac{t}{\theta})^2}{f(\frac{t}{\theta})^2} - 2\frac{f''(\frac{t}{\theta})}{f(\frac{t}{\theta})}\right) = \frac{1}{\theta^2} \sigma (f(\frac{t}{\theta})). $$
\end{observation}

\begin{proof}[Proof of Proposition \ref{prop:main2}]
Consider a diagram as in \eqref{diag:adjustingwidths}, with the properties provided by Proposition \ref{prop:flatten}, and furthermore, by Observation \ref{firstobsection}, we may assume that the metrics $G(0,x)$ is given by a warping function $f_x: [0,R_\infty] \to \bR$, with $R_\infty$ as in \ref{firstobsection}. The function $f_x$ is constant on $[R,R_\infty]$, and it is convenient to extend it to all of $[0,\infty)$, by $f(t):= f(R_\infty)$ for $t>R_\infty$. 

The goal is to deform $f_x$ into a function which is equal to the torpedo function $h_\delta$ on $[0,R]$, while retaining this property if $\norm{x}=1$. 

By \ref{secondobsection}, there exists $A>0$, so that 
\[
\sigma ((1-\lambda)f_x+\lambda h_\delta) \geq A
\]
for all $(x,\lambda )\in D^n \times [0,1]$. It follows from \ref{thirdobsection} that there is a continuous function $\theta:D^n \to (0,1]$ such that 
\[
\sigma ((1-\lambda)f^{\theta(x)}_x+\lambda h^{\theta(x)}_\delta)\geq B''
\]
for all $\lambda \in [0,1]$, $x\in D^n$ and $t\in [0,R]$. Since $f_x = h_\delta$ when $\norm{x} \geq \frac12$, we can furthermore assume that $\theta(x)=1$ for $\norm{x}=1$. 

The desired deformation of warping functions is given on the interval $[0,R]$ by the formula (note that $h_\delta^{\theta}=h_{\theta \delta }$)
\[
f_{x,\lambda}(t):= 
\begin{cases}
f_x^{3 \lambda \theta(x) + 1-3\lambda} (t)& \lambda \in [0, \frac13]\\ 
(2-3\lambda)f_x^{ \theta(x)}(t) + (3\lambda -1)h_{\delta_x}^{ \theta (x)}(t) & \lambda \in [ \frac13,\frac23]\\ 
h_{(3-3\lambda)\theta(x)\delta_x+ (3\lambda-2)\delta}(t) & \lambda \in [ \frac23,1].\\ 
\end{cases}
\]
Now, let 
\[
\beta:=\min_{(x,\lambda)\in D^n\times[0,1]}\{f_{x,\lambda}(R),f_x(R)\}\le\delta.
\]
and pick $R_\infty>R$ large enough so that there exists an $a$ as in \ref{firstobsection}. Finally, on the interval $[R,R_\infty]$ we define 
\[
f_{x,\lambda} (t)= a_{f_{x,\lambda}(R),f_x(R)}(t).\qedhere
\]
\end{proof}

\section{The fibration theorem and Theorem \ref{thm:main-boundary}}

\subsection{Proof of the fibration theorem}\label{subsec:fibrationtheorem}

We now present the proof of Theorem \ref{thm:improved-chernysh-theorem}.

\begin{lem}\label{lem:chernysh-thm-proof-approximation}
Let $M$ be a closed manifold, $P$ a compact space and $G:P \times [0,1] \to \Riem^+ (M)$ be continuous. Then there is a continuous map 
\[
 C: P \times [0,1] \times [0,1] \to \Riem^+ (M), (p,s,t) \mapsto C(p,s,t)
\]
such that
\begin{enumerate}
\item $C$ is smooth in $t$-direction,
\item all derivatives in $M$- and $t$-direction are continuous,
\item for all $(p,t) \in P \times [0,1]$, we have $C(p,0,t)=G(p,0)$,
\item for all $(p,s) \in P \times [0,1]$, we have $C(p,s,0)= G(p,0)$, and $C(p,s,1)= G(p,s)$,
\item if $K \subset M$ is a codimension $0$ submanifold and $G(p,t)|_{K}$ is independent of $t$, then $C(p,s,t)|_K$ is independent of $s$ and $t$.
\end{enumerate}
Moreover, there is a continuous function $\Lambda: [0,1] \to (0, \infty]$ with $\Lambda (0)=\infty$, such that if $k: \bR \to [0,1]$ is a smooth function with $|k'|, |k''| \leq \Lambda (s)$, then the metric
\[
 dt^2 + C(p,s,k(t))
\]
on $\bR \times M$ has positive scalar curvature for all $p \in P$.
\end{lem}

\begin{proof}
To construct $C$, let $U_{ni}= (\frac{i-1}{n}, \frac{i+1}{n})  \cap [0,1]$ and let $\cU_n$ be the open cover $(U_{ni})_{i=0, \ldots,n}$. 
%$\{ [0,\frac{1}{n}), (0,\frac{2}{n}), \ldots , (\frac{n-2}{n},1), (\frac{n-1}{n},1]\}$ of $[0,1]$ and l
Let $(\lambda_{ni})_{i=0,\ldots, n}$ be a subordinate smooth partition of unity. Define
\[
C_n (p,s,t):= \sum_{i=0}^n G(p,s \frac{i}{n}) \lambda_{ni} (t) \in \Riem (M).
\]
This has all the desired properties, except that the scalar curvature of $C_n (p,s,t)$ is not necessarily positive. A routine application of compactness proves that $\lim_{n \to \infty} C_n (p,s,t)= G(p,st)$, uniformly in all variables (the target space has the Fr\'echet topology, as usual). Since $\Riem^+ (M) \subset \Riem (M)$ is open, we find that for sufficiently large $n$, the scalar curvature of $C_n (p,s,t)$ is positive for all $p,s,t$. Define $C:= C_n$ for such an $n$.  

The existence of the function $\Lambda$ with the asserted property follows from the properties of $C$, from the compactness of $P$ and from Lemma \ref{lem:scalar-curvature-of-trace}.
\end{proof}

\begin{proof}[Proof of Theorem \ref{thm:improved-chernysh-theorem}]
Let $P$ be a disc and consider a lifting problem
\begin{equation}\label{chernysh-theorem-proof-liftingproblem}
\xymatrix{
P \times 0 \ar[d] \ar[r]^{F} & \Riem^+ (W) \ar[d]^{\res}\\
P \times [0,1] \ar[r]^{G} \ar@{..>}[ur]^{H} & \Riem^+ (M).
}
\end{equation}
Since $P$ is compact, we find $\delta>0$ such that $F (P \times 0) \subset \Riem^+ (W)^{2 \delta}$, by the observation made in the proof of Lemma \ref{lem:cont-bijection-equivalence}. 
We will construct a continuous map $K: P \times [0,1] \to \Riem^+ ([0,\delta] \times M)$ with the properties that 
\begin{enumerate}
 \item $K(p,s)= G(p,s)$ near $0 \times M$,
 \item $K(p,s)= G(p,0)$ near $\delta \times M$.
\end{enumerate}
Then define $H(p,s) \in \Riem^+ (W)$ to be equal to $K(p,s)$ on $[0, \delta] \times M$ and equal to $F(p)$ on $W \setminus ([0, \delta] \times M$. This is a solution to the lifting problem \eqref{chernysh-theorem-proof-liftingproblem}.

Let $C: P \times [0,1] \times [0,1] \to \Riem^+ (M)$ and $\Lambda: [0,1] \to (0,\infty]$ be as in Lemma \ref{lem:chernysh-thm-proof-approximation}. Now let 
\[
  a(s):= \max (\sqrt{\frac{3}{\Lambda(s)}}, \frac{3}{\Lambda(s)} ) +1 ,
\]
\[
 b:= \sup_s a(s),
\]
and fix a smooth function $f: \bR \to [0,1]$ such that 
\begin{enumerate}
 \item $f=0$ near $(-\infty,0]$, $f=1$ near $[1,\infty)$,
 \item $|f'|, |f''| \leq 3$.
\end{enumerate}
With these choices, the Riemannian metric 
\begin{equation}\label{chernyshtheorem-formula}
L (p,s):= dt^2 + C(p,s,f(\frac{t}{a(s)}))
\end{equation}
on $\bR \times M$ has positive scalar curvature, is cylindrical near $(-\infty,0] \times M$ and near $[b,\infty)\times M$ and lies in $\Riem^+ ([0,b] \times M)_{G(p,0), G(p,s)}$. Choose a diffeomorphism $h: [0, \delta] \to [0,b]$ such that $h'>0$, $h'=1$ near $0$ and $\delta$. Then $(p,s) \mapsto K(p,s) := (h \times \id_M)^* L(p,s)$ is the desired family of psc metrics on $[0,\delta] \times M$.
\end{proof}

\subsection{Proof of Theorem \ref{thm:main-boundary}}\label{sec:maintheoremboundary}

The proof uses an auxiliary construction. For $r\geq 0$, we let $W_r:= W \cup (M \times [0,r])$ be the result of gluing an external collar of length $r$ to $W$ and define $N_r$ analogously. We extend the metric $g_N$ on $N$ to one denoted $g_{N,r}$ on $N_r$ cylindrically. The embedding $\varphi: N\times\bR^k \to W$ gets extended in the obvious way to an embedding $\varphi_r: N_r \times \bR^{k} \to W_r$. We let $\Riem^+ (W_r,\varphi_r) \subset \Riem^+ (W_r)$ be the space of psc metrics which are of the form $g_{N,r}+g_\torp^k$ on $\varphi_r (N_r \times B_R^k)$. Extending psc metrics cylindrically over $M \times [0,r]$ defines maps $\Riem^+ (W_r) \to \Riem^+ (W_s)$ and $\Riem^+ (W_r,\varphi_r) \to \Riem^+ (W_s,\varphi_s)$ for $s >r$. Restriction to the boundary of $W_r$ defines restriction maps $\res': \Riem^+ (W_r) \to \Riem^+ (M)$ and $\Riem^+ (W_r,\varphi_r) \to \Riem^+ (M,\partial \varphi)$ which are compatible. In the colimit, we obtain a commutative diagram
\begin{equation}\label{diag:proof-of-bounded-cases}
\xymatrix{
\colim_{r \to \infty} \Riem^+ (W_r,\varphi_r) \ar[d]^{\colim_{r \to \infty }\res'} \ar[r] & \colim_{r \to \infty} \Riem^+ (W_r) \ar[d]^{\colim_{r \to \infty }\res'}\\
\Riem^+ (M,\partial \varphi) \ar[r] & \Riem^+ (M).
}
\end{equation}
\begin{lem}\label{lem:proof-of-oundedcases}
The left vertical map in \eqref{diag:proof-of-bounded-cases} is a Serre fibration. 
\end{lem}
Let us postpone the proof of Lemma \ref{lem:proof-of-oundedcases} for the moment, and explain how to finish the proof of Theorem \ref{thm:main-boundary}.

\begin{proof}[Proof of Theorem \ref{thm:main-boundary}]
The bottom horizontal map in \eqref{diag:proof-of-bounded-cases} is a weak equivalence, by Theorem \ref{thm:main}. The inclusion maps $\Riem^+ (W_r)\to \Riem^+ (W_s)$ are weak homotopy equivalences, and hence so is the inclusion map $\Riem^+ (W) \to \colim_{r\to \infty} \Riem^+ (W_r)$. This is proven in the same way as the elementary \cite[Lemma 2.1 and Corollary 2.3]{BERW}.

Each individual map $\Riem^+ (W_r,\varphi_r)\to \Riem^+ (W_r)$ is a weak equivalence, by Theorem \ref{thm:main} and hence so is the top horizontal map in \eqref{diag:proof-of-bounded-cases}. It follows that the inclusion $\Riem^+ (W,\varphi) \to \colim_{r \to \infty} \Riem^+ (W_r,\varphi_r)$ is a weak equivalence. 

The right vertical map in \eqref{diag:proof-of-bounded-cases} is a Serre fibration. This follows easily from Theorem \ref{thm:improved-chernysh-theorem} and a colimit argument.

Together with Lemma \ref{lem:proof-of-oundedcases}, it follows that the induced map on fibres is a weak equivalence. Over $g \in \Riem^+ (M,\partial \varphi)$, this is the bottom map of the diagram
\begin{equation}\label{diag:proof-of-bounded-cases2}
\xymatrix{
\Riem^+ (W,\varphi)_g \ar[r]\ar[d] & \Riem^+ (W)_g \ar[d]\\
\colim_{r \to \infty} \Riem^+ (W_r,\varphi_r)_g \ar[r] & \colim_{r \to \infty} \Riem^+ (W_r)_g,
}
\end{equation}
and we want to know that the top map is a weak equivalence. Thus to conclude the proof of Theorem \ref{thm:main-boundary}, it enough to prove that the two vertical maps in \eqref{diag:proof-of-bounded-cases2} are weak equivalences. For the right hand vertical map, this follows immediately from \cite[Lemma 2.1]{BERW}, and for the left hand side map, we use \cite[Corollary 2.5.4]{ERW17} (whose proof is elementary, but slightly lengthy). 
\end{proof}

\begin{proof}[Proof of Lemma \ref{lem:proof-of-oundedcases}]
This is similar to, but easier than the proof of Theorem \ref{thm:improved-chernysh-theorem}. Let $P$ be a disc and consider a lifting problem
\[
\xymatrix{
P \times 0 \ar[d] \ar[r]^-{F} & \colim_{r \to \infty} \Riem^+ (W_r,\varphi_r) \ar[d]^{\colim_{r \to \infty }\res'} \\
P \times [0,1] \ar[r]^-{G} \ar@{..>}[ur]^{H} & \Riem^+ (M,\partial \varphi) .
}
\]
By compactness of $P$, there is $r\geq 0$ so that $F(P \times 0) \subset \Riem^+ (W_r,\varphi_r)$. Define $L(p,s)$ by the formula \eqref{chernyshtheorem-formula} and let $b$ as in loc.cit. Define $H(p,s) \in \Riem^+ (W_{r+b},\varphi_{r+b})$ to be equal to $L(P,s)$ on $M \times [r,r+b]$ and equal to $F(p,s)$ on $W_r$. This has all the desired properties.
\end{proof}

\section{Cobordism invariance of the space of psc metrics}\label{sec:coboridmsection}

This section is devoted to the proof of Theorem \ref{thm:bordismapplication}. 

\begin{defn}
Let $M_0$ and $M_1$ be closed $(d-1)$-manifolds and let $W:M_0 \leadsto M_1$ be a cobordism. We say that $W$ has \emph{handle type $[k,l]$} if there exists a handlebody decomposition of $W$ relative to $M_0$ all whose handles have index in $\{k, k+1,\ldots,l-1,l\}$.  
\end{defn}

The first step in the proof of Theorem \ref{thm:bordismapplication} is the following corollary of Theorem \ref{thm:main}.

\begin{cor}[to Theorem \ref{thm:main}]\label{cor:handletype-weak}
Let $M_0$ and $M_1$ be two closed $(d-1)$-manifolds and assume that there is a cobordism $W:M_0 \leadsto M_1$ of handle type $[0,d-3]$. Then there exists a map $\Riem^+ (M_0) \to \Riem^+ (M_1)$, which can be chosen to be a weak homotopy equivalence if the handle type of $W$ is $[3,d-3]$.
\end{cor}

\begin{proof}
Let $M$ be a closed $(d-1)$-manifold and let $\phi: S^{m-1} \times \bR^{d-m} \to M$ be an embedding. The result of a surgery on $M$ is $M_\phi:= M \setminus (\phi(S^{m-1} \times D^{d-m}) \cup D^m \times S^{d-m-1}$. The opposite embedding of $\phi$ is $\phi^{op}: \bR^m \times S^{d-m-1} \to M_\phi$, and we obtain $M$ back from $M_\phi$ by a surgery on $\phi^{op}$. There is a zigzag of maps
\[
 \Riem^+ (M) \leftarrow \Riem^+ (M,\phi) \cong \Riem^+ (M_\phi;\phi^{op}) \to \Riem^+ (M_\phi).
\]
By Theorem \ref{thm:main}, the arrow pointing to the left is a weak equivalence if $m \leq d-3$. By inverting the arrow up to homotopy, we obtain a map $\Riem^+ (M) \to \Riem^+ (M_\phi)$. If $m \geq 3$, then the arrow pointing to the right is a weak equivalence, and so is the map $\Riem^+ (M) \to \Riem^+ (M_\phi)$. 

If $W$ has handle type $[0,d-3]$, we can obtain $M_1$ from $M_0$ by a sequence of surgeries on embedded copies of $S^{m-1} \times \bR^{d-m}$ with $d-m \geq 3$, which gives the map, and if the handle type is $[3,d-3]$, this map is a weak homotopy equivalence.
\end{proof}

Because of Corollary \ref{cor:handletype-weak}, Theorem \ref{thm:bordismapplication} follows from the next two results.

\begin{prop}\label{prop:get2-connectedcobordism}
Let $\theta: B \to BO(d)$ be a fibration and let $M_0$, $M_1$ be two closed $(d-1)$-dimensional $\theta$-manifolds such that the structure maps $M_i \to B$ are $2$-connected. Assume that $d \geq 6$ and that there is a $\theta$-cobordism $W:M_0 \leadsto M_1$. Then there exists a $\theta$-cobordism $W': M_0 \leadsto M_1$ such that both inclusions $M_i \to W'$ are $2$-connected.
\end{prop}

\begin{prop}\label{prop:handletrading}
Let $W: M_0 \leadsto M_1$ be a $d$-dimensional cobordism between closed manifolds such that the inclusions $M_i \to W$ are $2$-connected. Then
\begin{enumerate} 
\item if $d \geq 7$, $W$ has handle type $[3,d-3]$,
\item if $d=6$, then there is $r$ such that the connected sum $W\sharp^r (S^3 \times S^3)$ with sufficiently many copies of $S^3 \times S^3$ has handle type $[3,3]$. 
\end{enumerate}
\end{prop}

In both Propositions, the case $B = B \Spin $ follows quickly from the proof of the $h$-cobordism theorem for simply connected manifolds, which has found its way into textbooks \cite{Milnor}, \cite{Kos}, and is described in \cite{WalshC}. The general case requires techniques from the proof of the $s$-cobordism theorem, which are not so well-known. Therefore, we include the proof here, for sake of completeness.

\begin{proof}[Proof of Proposition \ref{prop:get2-connectedcobordism}]
To simplify the notation, we assume that $B$ is $0$-connected. 
If we could assume that the space is of type $(F_3)$ \cite{WallFin}, we can perform $\theta$-surgeries in the interior of $W$, giving a cobordism $W'$ such that the structure map $W' \to B$ is $3$-connected. It follows that the inclusion maps $M_i \to W$ are both $2$-connected. However, the assumptions of the proposition only imply that $B$ is of type $(F_2)$, and an additional argument is required (which goes back to \cite[Theorem 2.2]{RosNovII} and is explained also in \cite{HebJoa}).
By $\theta$-surgeries in the interior of $W$, we can replace $W$ by a cobordism $W'$ so that $W' \to B$ is $2$-connected. Hence we can assume, without loss of generality, that $W\to B$ is $2$-connected. This condition implies that both inclusions $M_i \to W$ induce isomorphisms on fundamental groups. Let $\pi$ be the common fundamental group. The long exact homotopy sequence of the pair $(W,M_0)$ gives the diagram
\[
 \xymatrix{
 \pi_2 (M_0)\ar[r] \ar@{->>}[dr] & \ar[d]^{(\ell_W)_\ast}\pi_2 (W) \ar[r] & \pi_2 (W,M_0)\ar[r] & 0 \\
  & \pi_2 (B) & & 
 }
\]
Now $\pi_2 (W,M_0)$ is a finitely generated $\bZ[\pi]$-module, by \cite[\S 1]{WallFin}. Using that the structure map $M_0 \to B$ is $2$-connected, a diagram chase shows that there are elements $x_1, \ldots,x_r \in \pi_2 (W)$ whose image in $\pi_2 (W,M_0)$ generate $\pi_2(W,M_0)$ as a $\bZ[\pi]$-module and which lie in the kernel of $(\ell_W)_\ast:\pi_2 (W) \to \pi_2 (B)$. 

Because $d \geq 5$, we can represent $x_i$ by an embedded $2$-sphere. Because $(\ell_W)_\ast (x_i)=0$, it follows that the normal bundle of those embeddings is trivial, and that $x_i$ is represented by an embedding $S^{2} \times \bR^{d-2} \to W$. We may assume that those embeddings are disjoint, by general position. We can perform $\theta$-surgeries on those spheres, and obtain a new $\theta$-cobordism $W'$ so that $M_0 \to W'$ is $2$-connected. Rename $W:=W'$.

To make the inclusion of $M_1$ $2$-connected as well, we use the same argument with $M_0$ replaced by $M_1$, but we have to be careful in order to not destroy the $2$-connectivity of $M_0 \to W$. For an embedding $f: \coprod^r S^2 \times D^{d-2} \to W$, let $W^{\circ}:= W \setminus \im (f)^{\circ}$ and $W'= W^{\circ} \cup_{\partial\im (f)} (\coprod^r D^3 \times S^{d-3})$. By general position, the inclusion $W^{\circ} \to W'$ is $2$-connected, and $W^{\circ} \to W$ is $(d-3) \geq 3$-connected (here we use that $d \geq 6$). The diagram
\[
 \xymatrix{
 W & \ar[l]_{3-\text{co.}} W^{\circ} \ar[r]^{2-\text{co.}} & W'\\
 & M_0\ar[ul]^{2-\text{co.}} \ar[u] \ar[ur]
 }
\]
shows that $M_0 \to W^{\circ}$ is $2$-connected, and so is $M_0 \to W'$
\end{proof}

%(the connectivities of the horizontal maps are proven by general position), and so if $M_0 \to W$ is $2$-connected, then so is $M_0 \to W'$. If $d \geq 6$, we are done.

\begin{proof}[Proof of Proposition \ref{prop:handletrading}]
Part (1) follows quickly from handle trading \cite{Kerv}, \cite{Wall}. If $d=6$, handle trading implies that $W$ is of handle type $[2,3]$. But handle trading is not enough: we can only barter the $2$-handles for $4$-handles. In fact, it is not correct that a $6$-dimensional cobordism $W:M_0 \leadsto M_1$ with both inclusions $2$-connected has handle type $[3,3]$: look at a $6$-dimensional $h$-cobordism with nontrivial Whitehead torsion to see why.
To deal with part (2), we invoke the following result, whose proof is contained in the proof of \cite[Lemma 6.21]{GRW} (and which has its origins in \cite{Kreck}). Let $W: M_0 \leadsto M_1$ be a cobordism of dimension $2n \geq 6$ such that both inclusions $M_i \to W$ are $(n-1)$-connected. Then for sufficiently large $r$, the cobordism $W \sharp^r (S^n \times S^n)$ has handle type $[n,n]$. 
\end{proof}

\bibliographystyle{plain}
\bibliography{gromovlawson}
\end{document}